\documentclass{article}%
\usepackage{amssymb}
\usepackage{amsfonts}
\usepackage{amsmath}
\usepackage{graphicx}%
\providecommand{\U}[1]{\protect\rule{.1in}{.1in}}
\newtheorem{theorem}{Theorem}
\newtheorem{acknowledgement}[theorem]{Acknowledgement}

\newtheorem{conjecture}[theorem]{Conjecture}
\newtheorem{corollary}[theorem]{Corollary}

\newtheorem{definition}[theorem]{Definition}
\newtheorem{example}[theorem]{Example}

\newtheorem{lemma}[theorem]{Lemma}

\newtheorem{proposition}[theorem]{Proposition}
\newtheorem{remark}[theorem]{Remark}

\newenvironment{proof}[1][Proof]{\noindent\textbf{#1.} }{\ \rule{0.5em}{0.5em}}
\begin{document}

\title{Weakly Chained Spaces}
\author{Conrad Plaut\\Department of Mathematics\\The University of Tennessee\\Knoxville TN 37996\\cplaut@utk.edu}
\maketitle

\begin{abstract}
We introduce \textquotedblleft weakly chained spaces\textquotedblright, which
need not be locally connected or path connected, but for which one has a
reasonable notion of generalized fundamental group and associated generalized
universal cover. We show that in the compact metric case, weakly chained is
equivalent to the concept of \textquotedblleft pointed
1-movable\textquotedblright\ from classical shape theory. We use this fact and
a theorem of Geoghegan-Swenson to give criteria on the metric spheres in a
CAT(0) space that imply that the boundary is has semistable fundamental group
at infinity.

\end{abstract}

\section{Introduction}

In this paper we introduce \textit{weakly chained spaces}, a very large class
of spaces with examples from so-called \textquotedblleft wild
topology\textquotedblright\ but also arising as boundaries of CAT(0) spaces
that satisfy certain \textquotedblleft infinitesimal\textquotedblright%
\ assumptions (\cite{PCAT}). In fact, in the compact metric case, we use a
result of Brodskiy, Dydak, Labuz and Mitra (\cite{Detal}) to prove that weakly
chained is equivalent to the notion of \textquotedblleft pointed
1-movable\textquotedblright\ from classical shape theory (Remark
\ref{pointed}). On the other hand, Geoghegan-Swenson (\cite{G}, Theorem 3.1)
showed that a $1$-ended proper CAT(0) space has semistable fundamental group
at infinity if and only if the boundary is pointed $1$-movable. It is a
long-standing conjecture of Mihalik (\cite{MSSI}) that every finitely
presented group $G$ is semi-stable at infinity, which is still open even for
CAT(0) groups. In light of the present paper and \cite{G}, one may check this
by verifying that the boundary of a CAT(0) space on which $G$ acts properly
and cocompactly by isometries is weakly chained. Theorem \ref{CAT} below
provides a method to verify this.

Because we are using theorems of others, we do not need to provide the
definition of \textquotedblleft pointed 1-movable\textquotedblright%
\ (originally due to Borsuk but mostly used in an equivalent form developed by
Marde\v{s}i\'{c}-Segal (\cite{MS})). We do note that it presupposes an
embedding in an ANR (e.g. the Hilbert Cube) and requires some work to verify,
for example, that the property does not depend on the ANR in question. In
contrast, for metric spaces \textquotedblleft weakly chained\textquotedblright%
\ can be defined in a single paragraph using using only the definition of
\textquotedblleft metric space\textquotedblright--as we will now do.

An $\varepsilon$-chain in a metric space is a finite sequence $\alpha
=\{x_{0},...,x_{n}\}$ of points such that for all $i$, $d(x_{i},x_{i+1}%
)<\varepsilon$. An $\varepsilon$-homotopy between $\alpha$ and another
$\varepsilon$-chain $\beta$ is a finite sequence of $\varepsilon$-chains
$\eta=\{\alpha=\alpha_{0},...,\alpha_{m}=\beta\}$ all with the same endpoints
such that each $\alpha_{i}$ differs from $\alpha_{i+1}$ by adding or removing
a single point (excluding endpoints). The $\varepsilon$-homotopy equivalence
class of $\alpha$ is denoted by $[\alpha]_{\varepsilon}$. A metric space $X$
is called \textit{chain connected} if for every $x,y\in X$, $x$ and $y$ may be
joined by an $\varepsilon$-chain for all $\varepsilon>0$. For simplicity we
will replace phrases like the last one by \textquotedblleft$x$ and $y$ may be
joined by arbitrarily fine chains\textquotedblright. A metric space $X$ is
called \textit{weakly chained} if it is chain connected and for every
$\varepsilon>0$ there is a $\delta>0$ such that if $d(x,y)<\delta$ then $x,y$
may be joined by arbitrarily fine chains $\alpha$ such that $[\alpha
]_{\varepsilon}=[x,y]_{\varepsilon}$; that is, $\alpha$ is $\varepsilon
$-homotopic to the two-point chain consisting of its endpoints.

A great variety of metric spaces are weakly chained--in fact the main examples
of compacta that are not weakly chained are solenoids (Example \ref{solenoid}%
). Among compact weakly chained spaces are all \textquotedblleft
sink-free\textquotedblright\ metric spaces (Proposition \ref{LSF}), defined as follows:

\begin{definition}
\label{sf}Let $X$ be a metric space. A (distance) sink in $X\times X$ is a
pair $(x,y)$ not on the diagonal at which the distance function has a
(possibly not strict) local minimum. If $X$ is chain connected and every sink
$(x,y)$ in $X$ satisfies $d(x,y)\geq\sigma$ for some $\sigma>0$ then $X$ is
called locally sink-free for $\sigma$, or LSF($\sigma$). When $\sigma$ is not
specified we simply call $X$ locally sink-free. If $X$ is LSF($\infty$) then
we say $X$ is sink-free.
\end{definition}

Put another way, if $x\neq y$ then $(x,y)$ is not a sink if and only if there
exist points $x^{\prime},y^{\prime}$ arbitrarily close to $x,y$ such that
$d(x^{\prime},y^{\prime})<d(x,y)$. That is, it is always possible to move the
points closer to one another with arbitrarily small movement. Examples of
sink-free metrics include all length, hence geodesic metrics (Example
\ref{length}); the Topologist's Sine Curve (Example \ref{tsc}); Euclidean
cones and spherical suspensions of arbitrary metric spaces (this follows
easily from an understanding of the geodesics in these spaces, see for example
\cite{BH}); and (obviously) any dense subspace of a sink-free metric space,
which means that sink-free metric spaces can be totally disconnected.

\begin{conjecture}
Every compact, weakly chained metric space admits a sink-free metric.
\end{conjecture}

Since geodesic metrics are sink-free, this conjecture can be considered as a
weak extension (pun unavoidable) of a question posed by Menger in 1928
(\cite{Men}), namely (in modern terminology) does every Peano Continuum admit
a geodesic metric? Menger's question was answered positively by the Bing-Moise
Theorem (\cite{B}, \cite{Moi}) in 1949. Note that it is easy to check that any
LSF($\sigma$) metric admits a sink-free metric that agrees with it when
distances are less than $\sigma$: just define the new distance between $x$ and
$y$ to be the infimum of lengths of $\sigma$-chains joining them.

Returning to CAT(0) spaces, see \cite{BH} for the standard terminology used in
the next theorem. For fixed $x_{0}\in X$, let $\Sigma_{x_{0}}(r):=\{y\in
X:d(x,y)=r\}$ denote the metric $r$-sphere at $x_{0}$.

\begin{theorem}
\label{CAT}Let $X$ be a proper, geodesically complete CAT(0) space and
$x_{0}\in X$. Suppose there exist some $K>0$ and a positive real function
$\iota$, called the refining increment, such that for all sufficiently large
$t$,

\begin{enumerate}
\item $\underset{s\rightarrow t^{+}}{\lim}\iota(s)>0$ (in particular if
$\iota$ is lower semicontinuous from the right) and

\item if $d(x,y)<\iota(t)$ and $(x,y)$ is a sink in $\Sigma_{x_{0}}(t)$ then
$x,y$ may be joined by a curve in $X_{t}\cap B(x,K)\cap B(y,K)$. 
\end{enumerate}

Then $\partial X$ is weakly chained, hence $X$ is semi-stable at infinity.
\end{theorem}

In \cite{PCAT} we give \textquotedblleft infinitesimal\textquotedblright\ and
relatively easy to verify conditions under which Theorem \ref{CAT} can be
applied. One convenience of Theorem \ref{basis} is that discrete homotopy
theory has been pushed into the background and is not involved in the hypotheses.

Even if one is only interested in metric spaces, the most appropriate setting
for these ideas is uniform spaces--especially if there is no natural metric,
as is true with some inverse limits. We will provide more details later, but
briefly a uniform space $X$ is a topological space together with a collection
of symmetric subsets of $X\times X$ containing an open subset containing the
diagonal called \textit{entourages}, which\textit{ }satisfy the following
\textquotedblleft triangle inequality\textquotedblright\ property: For every
entourage $E$ there is an entourage $F$ such that $F^{2}\subset E$. Here
$F^{2}$ is the set of all $(x,y)$ such that for some $w$, $(x,w),(w,y)\in F$.
A collection $\mathcal{B}$ of entourages such that every entourage contains an
element of $\mathcal{B}$ is called a \textit{uniformity basis} or simply
\textit{basis}. The main examples of uniform spaces are topological groups,
metric spaces (with a basis consisting of \textit{metric entourages}
$E_{\varepsilon}=\{(x,y):d(x,y)<\varepsilon\}$), and compact topological
spaces, which have a unique uniform structure with basis consisting of all
open subsets of $X\times X$ containing the diagonal. We always assume that
uniform spaces are Hausdorff, which is equivalent to the intersection of all
entourages being the diagonal of $X\times X$. One may now imitate the above
metric definitions using the uniform structure. For example, and $E$-chain is
a finite sequence $\alpha=\{x_{0},...,x_{n}\}$ of points such that for all
$i$, $(x_{i},x_{i+1})\in E$. In particular, an $\varepsilon$-chain is just an
$E_{\varepsilon}$-chain, and \textquotedblleft$x$ and $y$ may be joined by
arbitrarily fine chains\textquotedblright\ means that for every entourage $E$
there is an $E$-chain joining $x$ and $y$.

In 2001, Berestovskii-Plaut produced two ways to construct covering groups for
topological groups (\cite{BPTG}). The first method employed a construction
discovered independently by Schreier in 1925 (\cite{Sch}) and Mal'tsev in 1941
(\cite{Malcev}) to embed a local group into a uniquely determined topological
group. The Schreier-Mal'tsev construction is similar to the construction of a
finitely presented group: begin with the semigroup of finite words with
letters in the given local group, then mod out by an equivalence relation
determined by products that are defined in the local group. Berestovskii-Plaut
showed that if this construction is applied to a neighborhood of the identity
of a given topological group, one obtains a covering group of the original
group. Under appropriate circumstances, the inverse limit of these covering
groups is a kind of generalized universal cover. In the same paper,
Berestovskii-Plaut produced an equivalent method to produce covering groups
using discrete chains and homotopies, which does not need any underlying local
group structure and could be applied much more generally. The required notion
of generalized regular covering map for uniform spaces was developed by the
author in 2006 (\cite{Pcov}) and discrete homotopy theory for uniform spaces
was developed in \cite{BPUU}. Our second application in the present paper is
to extend the generalized universal covering space results of
\textit{\cite{BPUU} }not only to weakly chained uniform spaces, but also to
certain metrizable topological spaces, including all path connected metrizable
topological spaces.

The most basic construction in discrete homotopy theory for uniform spaces
goes as follows. Fixing a basepoint $\ast$, the set of all equivalence classes
$[\alpha]_{E}$ of $E$-chains starting at $\ast$ is denoted by $X_{E}$; the
endpoint mapping is denoted $\phi_{E}:X_{E}\rightarrow X$, which is surjective
if and only if every pair of points in $X$ may be joined by an $E$-chain; this
is obviously true if $X$ is chain connected. $X_{E}$ has a uniquely determined
uniform structure such that if $X$ is chain connected then $\phi_{E}$ is a
uniformly continuous regular covering map, although it is a very important
consideration that $X_{E}$ may not be chain connected. The group of covering
transformations of $\phi_{E}$ is naturally isomorphic to the group $\pi
_{E}(X)$ consisting of the $E$-homotopy equivalence classes of $E$-loops with
operation induced by concatenation. That is, $\pi_{E}(X)$ can be considered as
a kind of fundamental group \textquotedblleft at the scale of $E$%
\textquotedblright.

The covering maps $\phi_{E}:X_{E}\rightarrow X$ already have applications in
Riemannian geometry, including finiteness theorems (\cite{PW1}) and an
alternative approach to the Covering Spectrum of Sormani-Wei (\cite{SW2}),
which itself has numerous applications in geometric analysis. It is an open
question how to algebraically characterize these \textquotedblleft
entourage\textquotedblright\ covering maps (with some assumptions to ensure
connectedness of $X_{E}$), except for compact smooth manifolds of dimension
$d=1$ (the only non-trivial entourage cover is the universal cover) and
$d\geq3$ (entourage covers are precisely the regular covering maps
corresponding to the normal closures of finite sets in the fundamental group),
see \cite{PLS}.

When $F\subset E$ there is a natural mapping $\phi_{EF}:X_{F}\rightarrow
X_{E}$ defined by $\phi_{EF}([\alpha]_{F})=[\alpha]_{E}$. The restriction
$\theta_{EF}$ of $\phi_{EF}$ to $\pi_{F}(X)$ is a homomorphism into $\pi
_{E}(X)$, which is surjective or injective if and only if $\phi_{EF}$ is,
respectively. These mappings form an inverse system called the
\textit{fundamental inverse system} of $X$ and the inverse limit is denoted
$\widetilde{X}$. The maps $\theta_{EF}$ also form an inverse system of groups
with inverse limit denoted by $\pi_{U}(X)$ and called the \textit{uniform
fundamental group} of $X$. The construction of $\widetilde{X}$ can be carried
out for any uniform space, and the question becomes what are the properties of
$\widetilde{X}$, the natural projection $\phi:\widetilde{X}\rightarrow X$, and
$\pi_{U}(X)$? In analogy with the classical theory, and extending the notion
of \textquotedblleft universal space\textquotedblright\ from \cite{BPUU}, we
make the following definition.

\begin{definition}
If $X$ is a weakly chained uniform space such that $\pi_{U}(X)$ is trivial
then we will call $X$ \textit{uniformly simply connected}. If $f:Y\rightarrow
X$ is a generalized regular covering map and $Y$ is uniformly simply connected
then $f$ is called a uniform universal covering map (UU-cover) of $X$.
\end{definition}

\begin{theorem}
\label{main1}If $X$ is a metrizable weakly chained space then $\phi
:\widetilde{X}\rightarrow X$ is a UU-cover of $X$ with deck group $\pi_{U}(X)$.
\end{theorem}

For simplicity we will give only an intuitive definition of generalized
regular covering map as defined in \cite{Pcov} (perhaps confusingly just
called \textquotedblleft covers\textquotedblright\ in that paper): A
traditional regular covering map of a topological space may be roughly
described as \textquotedblleft the quotient map of a discrete action by a
group of homeomorphisms\textquotedblright. Analogously, a generalized regular
covering map of a uniform space may be roughly described as \textquotedblleft
the quotient map of a prodiscrete action by a group of uniform
homeomorphisms\textquotedblright. In either case, as is traditional in
geometry, we will call the group in question the \textquotedblleft deck
group\textquotedblright\ of the generalized regular covering map. We will use
the theorem from \cite{Pcov} that generalized regular covering maps are
precisely inverse limits of \textquotedblleft discrete\textquotedblright%
\ covering maps, the uniform analog of traditional regular covering maps.

Once the existence of the UU-cover is established, the following
\textquotedblleft classical\textquotedblright\ properties follow, which were
shown in \cite{BPUU} in the setting of \textquotedblleft coverable
spaces\textquotedblright. \textbf{Lifting:} If $Y$ is uniformly simply
connected, $X$ is weakly chained and $f:Y\rightarrow X$ is uniformly
continuous then there is a unique (up to basepoint) uniformly continuous map
$f_{L}:Y\rightarrow\widetilde{X}$ (called the lift of $f$) such that
$f=\phi\circ f_{L}$. \textbf{Universal: }If $g:Y\rightarrow X$ is a
generalized regular covering map between weakly chained spaces then $\phi$
factors through it, i.e. there is a unique (up to basepoint) generalized
regular covering map $h:\widetilde{X}\rightarrow Y$ such that $\phi=g\circ h$.
In particular, the UU-cover is unique up to uniform homeomorphism and choice
of basepoint. \textbf{Functorial: }If $f:Y\rightarrow X$ is uniformly
continuous between weakly chained spaces then there is a unique (up to
basepoint) uniformly continuous map $\widetilde{f}:\widetilde{Y}%
\rightarrow\widetilde{X}$ that commutes with $f$ and the respective UU-covers.
This map also induces a homormorphism $f_{\#}:\pi_{U}(X)\rightarrow\pi_{U}(Y)$
and satisfies $\widetilde{f\circ g}=\widetilde{f}\circ\widetilde{g}$.

To fully complement the classical theory, one may define a \textit{generalized
covering map} (\textit{sans }regular) between weakly chained spaces to be any
uniformly continuous surjection $f:X\rightarrow Y$ through which
$\phi:\widetilde{Y}\rightarrow Y$ factors, although we do not use this concept
in this paper.

Coverable,\ which is stronger than weakly chained (Corollary \ref{coverable}),
means that the projections $\phi^{E}:\widetilde{X}\rightarrow X_{E}$ are
surjective for all $E$ in some uniformity basis. Coverable is easy to verify
in come cases, for example for Peano continua, but in general it can be
problematic to verify because it involves the structure of the fundamental
inverse system. In fact inverse systems are one situation in which the
definition of weakly chained offers many advantages. The next example shows
that coverable is strictly stronger than weakly chained, but we do not know
whether these properties are equivalent for metrizable spaces.

\begin{example}
For locally compact topological groups with their unique invariant uniform
structures, Berestovskii-Plaut showed that coverable is equivalent to path
connected (Theorem 7, \cite{BPLGC}). On the other hand, the character group
$G$ of the discrete group $\mathbb{Z}^{\mathbb{N}}$ was shown by Dixmier in
1957 (\cite{Dx}) to be a compact (non-metrizable!), connected, locally
connected topological group that is not path connected. Therefore $G$ is not
coverable. But since $G$ is locally connected, as we will see below $G$ is
weakly chained. On the other hand, as shown in the proof of Theorem 9 in
\cite{BPLGC}, $\phi:\widetilde{G}\rightarrow G$ is not surjective and so
cannot be a UU-cover. This shows that \textquotedblleft
metrizable\textquotedblright\ cannot be removed from Theorem \ref{main1}.
Nonetheless, since $G$ is a compact, connected group, $G$ has a (compact!)
universal cover in yet another sense (\cite{BPCG}).
\end{example}

Every metrizable (more generally completely regular) topological space has a
unique finest uniform structure compatible with its topology, called the
\textit{fine uniformity. }This result springs out of some category theory in
\cite{I}, but we give a more concrete statement and proof for the metrizable
case in Proposition \ref{fine}. As is stated in \cite{I}, but as also
immediately follows from Proposition \ref{fine}, if topological spaces have
the fine uniformity, then maps between them are continuous if and only if they
are uniformly continuous. We call a topological space $X$ \textit{weakly
chained} if $X$ with the fine uniformity is weakly chained. We will use the
same terms (\textit{UU-cover} and \textit{uniform fundamental group}) to
describe, for a weakly chained topological space, the (now topological!)
invariants with the same names associated with the fine uniformity. 

In the same year that \cite{BPUU} appeared, Fischer and Zastrow (\cite{FZ})
also defined what they called a \textquotedblleft generalized universal
cover\textquotedblright\ for topological spaces, which we will distinguish
from ours by referring to it as the \textquotedblleft Fischer-Zastrow simply
connected cover\textquotedblright. They defined it to be a continuous function
$p:\widetilde{X_{FZ}}\rightarrow X$ where $\widetilde{X_{FZ}}$ is path
connected and simply connected, such that the following \textbf{Topological}
\textbf{Lifting} property holds: If $Y$ is a path connected, locally path
connected, simply connected space and $f:Y\rightarrow X$ is continuous then
there is a unique (up to basepoint) continuous lift $f_{L}:Y\rightarrow
\widetilde{X_{FZ}}$. They worked with more general topological spaces, but
showed that for metrizable spaces, if the classical homomorphism $\kappa
:\pi_{1}(X)\rightarrow\overset{\vee}{\pi}_{1}(X)$ from the fundamental group
into the shape group is injective (a property referred to as \textquotedblleft
shape injective\textquotedblright\ in \cite{CC}), then the Fischer-Zastrow
simply connected cover exists. At the same time, Berestovskii-Plaut showed in
\cite{BPUU}, Proposition 80, that if $X$ is path connected and a naturally
defined homomorphism $\lambda:\pi_{1}(X)\rightarrow\pi_{U}(X)$, is injective
then the UU-cover is simply connected! Moreover, in the compact metrizable
case, $\lambda$ is naturally identified with $\kappa$ (Remark \ref{pointed}).

Let $X$ be a metrizable, path connected topological space, and let $X^{LP}$
denote the locally path-connected co-reflection of $X$ (there is a nice
discussion of this concept, including otherwise unpublished results, in
\cite{Blog}). $X^{LP}$ has a finer topology than $X$, is metrizable if $X$ is
metrizable, and is path and locally path connected. If $X$ is already locally
path connected then $X$ is homeomorphic to $X^{LP}$. Moreover, if
$f:Y\rightarrow X$ is a continuous function with $Y$ path and locally path
connected, $f:Y\rightarrow X^{LP}$ is also continuous. By Proposition
\ref{loc2g}, when $X^{LP}$ is given the fine uniformity, it is coverable,
hence weakly chained. In addition, if $Y$ is path and locally path connected,
it follows from Proposition \ref{loc2g} and Theorem 72 in \cite{BPUU} that $Y$
with the fine uniformity is uniformly simply connected. Therefore, the
composition $\widetilde{X^{LP}}\rightarrow X^{LP}\rightarrow X$, which we will
also denote by $\phi$, has the Topological Lifting property, and any lift
$f_{L}$ must go into the path component of the basepoint of $\widetilde{X^{LP}%
}$. Therefore, if $\lambda$ is injective (so $\widetilde{X^{LP}}$ is simply
connected), the restriction of $\phi$ to the path component of
$\widetilde{X^{LP}}$ (which maps onto $X$ due to path lifting) is the
Fischer-Zastrow simply connected cover.

Conversely, suppose that $X$ is a metrizable, path connected topological space
with a Fischer-Zastrow simply connected cover. Then according to the
Topological Lifting property, there is a unique base-point preserving
uniformly continuous lift $\tau:\widetilde{X_{FZ}}\rightarrow\widetilde{X^{LP}%
}$, which maps into the path component of the basepoint of $\widetilde{X^{LP}%
}$, and which commutes with the two generalized universal covering maps. But
Topological Lifting property of the Fischer-Zastrow map produces an inverse
function to $\tau$. We obtain the following version of their theorem, with the
only remaining question being the relationship between $\kappa$ and $\lambda$
in the non-compact case (which we will leave for a shape theorist to work out).

\begin{theorem}
Let $X$ be a path connected metrizable topological space. If the map
$\lambda:\pi_{1}(X)\rightarrow\pi_{U}(X)$ is injective (e.g. if $X$ is compact
and shape injective) then the Fischer-Zastrow simply connected covering map of
$X$ exists and is naturally identified with the restriction of the UU-cover of
$X^{LP}$ to the path component of the identity.
\end{theorem}

From Theorem 79 in \cite{BPUU} we know that $\widetilde{X^{LP}}$ is path
connected if and only if the map $\lambda$ is surjective. In other words, if
$\lambda$ is an isomorphism then the Fischer-Zastrow simply connected cover
must be the UU-cover of $X^{LP}$. If $\lambda$ is not surjective, then in
choosing between these two generalized universal covers, one must decide
whether path connectedness or the regular structure and completeness of the
UU-cover is more appropriate for the given problem. The Hawaiian Earring $H$
is an example of a Peano continuum such that the map $\lambda$ is not
surjective (this is known to topologists for the map $\kappa$, but specific
calculations involving $\lambda$ for $H$ and other Peano continuua are given
in Section 7 of \cite{BPUU}). In general if $X$ is any Peano continuum then by
uniqueness, the UU-cover can be obtained by taking the completion as a uniform
space (\cite{I}) of the fine uniformity on the Fischer-Zastrow simply
connected cover, because the path component of $\widetilde{X}$ is dense
(\cite{BPUU}, Proposition 82).

There is another connection between these two constructions related to an
earlier construction of Sormani-Wei from 2001 (\cite{SW1}, \cite{SW2}) for
metric spaces. Fischer-Zastrow in their existence proof use a construction of
Spanier (\cite{S}) involving existence of covering maps that are determined by
open coverings of a space. Sormani-Wei used the same construction of Spanier
applied to the covering of a geodesic space by $\delta$-balls. Specifically,
call a path loop \textquotedblleft$\delta$-small\textquotedblright\ if it is
of the form $\alpha\ast\tau\ast\overline{\alpha}$, where the loop $\tau$ is
contained in an $\delta$-ball. Spanier's construction then provides a covering
map determined by the subgroup of the fundamental group generated by $\delta
$-small loops, which they called the $\delta$-cover of a metric space. Despite
the completely different construction, Plaut-Wilkins showed that for compact
geodesic spaces, the $\varepsilon$-cover $X_{\varepsilon}:=X_{E_{\varepsilon}%
}$ is equivalent to the Sormani-Wei $\delta$-cover when $\delta=\frac{3}%
{2}\varepsilon$ (\cite{PW2}). Sormani-Wei did not consider the inverse limit
of their $\delta$-covers (there are many applications of this construction to
Riemannian geometry without doing so), but the inverse limit, of course, must
be $\widetilde{X}$.

Since $\delta$-small path loops are trivially $\delta$-null in the sense of
\cite{PLS}, they lift as loops to $X_{\delta}$. It follows that if a path loop
$\lambda$ is homotopic to arbitrarily small loops, then it lifts as a loop to
every $X_{\delta}$ and hence as a loop to $\widetilde{X}$. In the compact
case, the contrapositive statement gives in an alternative proof that shape
injective spaces are homotopically Hausdorff in the sense of Cannon-Conner
(\cite{CC}). The converse is not true. Virk-Zastrow (\cite{VZ}) constructed a
path and locally path connected space $RX$ that is homotopically Hausdorff but
for which the Fischer-Zastrow construction does not produce a generalized
universal cover because uniqueness of path liftings fails. In other words,
from the above discussion, the space is not shape injective. But $RX$ is a
Peano continuum and hence coverable. Therefore the UU-covering space
$\widetilde{RX}$ nonetheless exists and is uniformly simply connected (but not
simply connected! cf. Proposition 80 in \cite{BPUU}). --And of course the
UU-cover has the Lifting, Universal, and Functorial properties described earlier.

\section{Background}

We first recall a few concepts about uniform spaces, sometimes with notation
not used by classical authors, among whom notation varies somewhat. To help
with the exposition and our notation, we will give a couple of proofs of very
basic concepts but claim no originality for those results. Basic statements
that we do not prove can be found in standard texts such as \cite{I}. Some
quick reminders: A uniform space has a compatible metric if and only if it has
a countable basis. The subspace uniformity on $A\subset X$ consists of the
intersections of entourages in $X$ with $A\times A$. What we call
\textquotedblleft chain connected\textquotedblright\ is equivalent to what is
known as \textquotedblleft uniformly connected\textquotedblright\ in the
classical literature and it is a basic result that for compact spaces,
connected and chain connected are equivalent. Like components, chain
components are closed but need not be open.

For a topological space, being \textquotedblleft
uniformizable\textquotedblright\ (having a uniform structure compatible with a
given topology) is equivalent to being completely regular. We do not know of a
direct proof in the literature of the next result:

\begin{proposition}
\label{fine}Let $X$ be a metrizable topological space. The collection of all
symmetric sets containing open sets containing the diagonal in $X\times X$ is
a uniform structure compatible with the topology that contains every uniform
structure compatible with the topology.
\end{proposition}

\begin{proof}
Let $X$ have any metric. Let $E$ be a symmetric open set containing the
diagonal in $X\times X$. For each $x\in X$ there is some $\varepsilon_{x}>0$
such that $B(x,\varepsilon_{x})\times B(x,\varepsilon_{x})\subset E$. Define
$F:=%
{\displaystyle\bigcup\limits_{x\in X}}
\left[  B(x,\frac{\varepsilon_{x}}{2})\times B(x,\frac{\varepsilon_{x}}%
{2})\right]  $. $F$ is clearly a symmetric open set containing the diagonal,
and we claim that $F^{2}\subset E$. If $(a,c)\in F^{2}$, this by definition
means that there is some $b\in X$ such that $(a,b)\in F$ and $(b,c)\in F$.
That is, there exist $x,y\in X$ such that $d(a,x),d(b,x)<\frac{\varepsilon
_{x}}{2}$ and $d(b,y),d(c,y)<\frac{\varepsilon_{y}}{2}$. Without loss of
generality, $\varepsilon_{x}\geq\varepsilon_{y}$. By the triangle inequality,
$d(c,x)<\frac{\varepsilon_{x}}{2}+\frac{\varepsilon_{y}}{2}\leq\varepsilon
_{x}$. Since $d(a,x)<\frac{\varepsilon_{x}}{2}<\varepsilon_{x}$, $(a,c)\in
B(x,\varepsilon_{x})\times B(x,\varepsilon_{x})\subset E$. The last statement
of the proposition is obvious, since entourages are by definition symmetric
sets containing open sets containing the diagonal.
\end{proof}

We define the $E$\textit{-ball} at $x$ to be $B(x,E):=\{y:(x,y)\in E\}$; in
metric spaces, $B(x,\varepsilon)=B(x,E_{\varepsilon})$. A subset $U$ of a
uniform space is called \textit{uniformly }$F$-\textit{open} for an entourage
$F$ if for any $x\in U$, $B(x,F)\subset U$. If $F$ is unspecified we will
simply call $U$ \textit{uniformly open}. It is easy to check that uniformly
open sets are both open and closed; in fact, if $U$ is uniformly $F$-open then
the complement of $U$ is clearly also uniformly $F$-open. The following lemma
is useful to understand these concepts.

\begin{lemma}
\label{10}Let $X$ be a uniform space, $E$ be an entourage, and $x\in X$. The
set
\[
U_{x}^{E}:=\{y\in X:\text{there is an }E\text{-chain from }x\text{ to }y\}
\]
is the smallest uniformly $E$-open set in $X$ containing $x$.
\end{lemma}

\begin{proof}
Suppose $y\in U_{x}^{E}$, i.e. there is an $E$-chain $\alpha=\{x=x_{0}%
,...,x_{n}=y\}$. Let $z\in B(y,E)$. Since $(y,z)\in E$, $\{x=x_{0}%
,...,x_{n}=y,z\}$ is an $E$-chain, showing $z\in U_{x}^{E}$ and hence that
$U_{x}^{E}$ is uniformly open. Now suppose that $V$ is a uniformly $E$-open
set containing $x$ that doesn't contain $y$. Since $x=x_{0}\in V$ there is
some $i$ such that $x_{i}\in V$ but $x_{i+1}\notin V$. Since $V$ is uniformly
$E$-open and $x_{i+1}\in B(x_{i},E)$, this is a contradiction.
\end{proof}

Note that by definition the chain component of $x$ is the intersection of all
the sets $U_{x}^{E}$. From the above lemma also follows the classical fact
(which we will use without reference) that $X$ is chain connected if and only
if the only uniformly open subsets of $X$ are $X$ and $\varnothing$.

A simple example to illustrate these concepts is the following: Let
$\gamma_{1}$ be the graph of $y=\frac{1}{x}$ and $\gamma_{2}$ be the $x$-axis,
and $X$ be the union of these sets, with the uniformity of the subspace metric
(which coincides with the subspace uniformity). The sets $\gamma_{i}$ are open
in $X$ but not uniformly open in $X$. $X$ is not connected, but is chain
connected, and in particular $X$ has two components but only one chain component.

We will abuse notation involving images and inverse images of subsets of
$X\times X$, for example writing $f(E)$ rather than $\left(  f\times f\right)
(E)$. In this notation one may take the definition of uniform continuity of
$f:X\rightarrow Y$ between uniform spaces to be that for any entourage $E$ in
$Y$, $f^{-1}(E)$ is an entourage in $X$. Equivalently, for every entourage $E
$ in $Y$ there is an entourage $F$ in $X$ such that $f(F)\subset E$. Following
\cite{BPUU}, we say that $f$ is \textit{bi-uniformly continuous} if $f$ is
uniformly continuous and for every entourage $E$ in $X$, $f(E)$ is an
entourage in $f(X)$. A bijective bi-uniformly continuous function is called a
\textit{uniform homeomorphism}. It is easy to check that if $f$ is uniformly
continuous, the inverse image (resp. image) of any uniformly open (resp. chain
connected) set is uniformly open (resp. chain connected).

We now very briefly recall the basics of discrete homotopy theory for uniform
spaces, continuing from the Introduction. Much of this is from \cite{BPUU},
but \cite{PLS} has additional results and uses our current notation. Also,
\cite{PW1} has an exposition in the more familiar setting of metric spaces.
Let $X$ be a uniform space and $E$ be an entourage. There are two basic
$E$-homotopy moves that are useful: adding or taking away a repeated point.
For example, by adding or taking away repeated points we can always assume, up
to $E$-homotopy, that two $E$-chains with the same endpoints have the same
number of points; one can then work with \textquotedblleft
corresponding\textquotedblright\ points, which is useful for some arguments.
This is also helpful when considering concatenation: if $\alpha$ and $\beta$
are $E$-chains such that the last point $y$ of $\alpha$ is the first point of
$\beta$, $\alpha\ast\beta$ denotes the concatenation of $\beta$ followed by
$\alpha$. Strictly speaking, the concatenation repeats $y$, but the duplicate
can be removed up to $E$-homotopy. We denote by $\overline{\alpha}$ the
reversal of $\alpha$. If $\alpha$ is an $E$-loop that is $E$-homotopic to its
start/end point then $\alpha$ is called $E$-null. We will denote
$[\{x_{0},...,x_{n}\}]_{E}$ simply by $[x_{0},...,x_{n}]_{E}$. As long as $X$
is chain connected, choice of basepoint does not matter up to uniform
homeomorphism (\cite{BPUU}), and we may always choose (and we will always
assume) the maps in these constructions to be basepoint preserving. For
example, if $\ast$ is the basepoint in $X$, we choose $[\ast]_{E}%
=[\{\ast\}]_{E}$ to be the basepoint of $X_{E}$. Note that by definition, the
image of $\phi_{E}$ is $U_{\ast}^{E}$ (see Lemma \ref{10}).

For any entourage $F\subset E$, define $F^{\ast}$ to be the set of all
$([\alpha]_{E},[\beta]_{E})$ such that $[\alpha\ast\overline{\beta}%
]_{E}=[a,b]_{E}$, where $a,b$ are the endpoints of $\alpha,\beta$,
respectively, and $(a,b)\in F$. This is equivalent to the slightly more
cumbersome definition in \cite{BPUU}. It is an easy exercise that
$([\alpha]_{E},[\beta]_{E})\in F^{\ast}$ if and only if there is an $F$-chain
$\gamma$ from $a$ to be such that $[\gamma]_{E}=[a,b]_{E}$. The set of all
such $F^{\ast}$ is the basis for the naturual uniform structure on $X_{E}$
that we will use. Moreover, the entourages $F^{\ast}$ are invariant with
respect to the action of $\pi_{E}(X)$ and $\phi_{E}$ is a uniform
homeomorphism restricted to any $B(x,F^{\ast})$ onto $B(\phi_{E}(x),F)$. Note
also that by definition $\phi_{E}(F^{\ast})\subset F$, and if $X$ is chain
connected then
\begin{equation}
\phi_{E}(F^{\ast})=F\text{.} \label{surj}%
\end{equation}
To see this, suppose $(x,y)\in F$. Since $X$ is chain connected, there is an
$E$-chain $\alpha$ from $\ast$ to $x$. Letting $\beta:=\alpha\ast\{x,y\}$,
$([\alpha]_{E},[\beta]_{E})\in F^{\ast}$ by definition, and $\phi_{E}%
([\alpha]_{E})=x$ and $\phi_{E}([\beta]_{E})=y$. If $X$ is not chain
connected, the same argument shows that
\begin{equation}
\phi_{E}(F^{\ast})=F^{con}:=F\cap(X^{c}\cap X^{c}) \label{surj2}%
\end{equation}
where $X^{c}$ is the chain component of $X$ containing the basepoint. Note
that by definition, the sets $F^{con}$ are a basis for the subspace uniformity
on $X^{c}$.

The first part of the next Lemma is from \cite{PLS} (it was hinted at but not
explicitly stated in \cite{BPUU}). The second numbered statement is the analog
of the homotopy lifting property from clasical covering space theory.

\begin{lemma}
[Chain Lifting]\label{cl}Let $X$ be a uniform space and $E$ be an entourage.
Suppose that $\beta:=\{x_{0},...,x_{n}\}$ is an $E$-chain and $[\alpha]_{E}$
is such that $\phi_{E}([\alpha]_{E})=x_{0}$. Let $y_{i}:=[\alpha\ast
\{x_{0},...,x_{i}\}]_{E}$. Then $\widetilde{\beta}:=\{y_{0}=[\alpha]_{E}%
,y_{1},...,y_{n}=[\alpha\ast\beta]_{E}\}$ is the unique \textquotedblleft
lift\textquotedblright\ of $\beta$ at $[\alpha]_{E}$. That is,
$\widetilde{\beta}$ is the unique $E^{\ast}$-chain in $X_{E}$ starting at
$[\alpha]_{E}$ such that $\phi_{E}(\widetilde{\beta})=\beta$. Moreover,

\begin{enumerate}
\item If $\beta$ is an $F$-chain for some entourage $F\subset E$ then
$\widetilde{\beta}$ is an $F^{\ast}$-chain.

\item If $[\beta]_{E}=[\gamma]_{E}$ then $[\widetilde{\beta}]_{E^{\ast}%
}=[\widetilde{\gamma}]_{E^{\ast}}$.
\end{enumerate}
\end{lemma}

\begin{proof}
Only the second numbered statement is new. By induction we need only verify
that if $\beta$ and $\gamma$ differ by a basic move, then the conclusion
holds. Let $\widetilde{\beta}:=\{\widetilde{x_{0}},...,\widetilde{x_{n}}\}$,
where $\phi_{E}(\widetilde{x_{i}})=x_{i}$. Suppose $\gamma=\{x_{0}%
,...,x_{j},x,x_{j+1},...,x_{n}\}$. By the first part of the lemma, the
$E$-chain $\{x_{j},x,x_{j+1}\}$ has a unique lift to an $E^{\ast}$-chain
$\kappa=\{\widetilde{x_{j}},\widetilde{x},\widetilde{z}\}$ starting at
$\widetilde{x_{j}}$. Since $[x_{j},x,x_{j+1}]_{E}=[x_{j},x_{j+1}]_{E}$, by the
first part of this lemma, $\kappa$ and the unique lift of $\{x_{j},x_{j+1}\}$
must end in the same point. That is, $\widetilde{z}=\widetilde{x_{j+1}}$.
Since $(\widetilde{z},\widetilde{x})\in E^{\ast}$, $(\widetilde{x}%
,\widetilde{x_{j+1}})\in E^{\ast}$. That is, adding $\widetilde{x}$ is a basic
move. By uniqueness, $\{\widetilde{x_{0}},...,\widetilde{x_{j}},\widetilde{x}%
,\widetilde{x_{j+1}},...,\widetilde{x_{n}}\}$ is the lift of $\gamma$.
Removing a point is simply the reverse operation, so the proof of the lemma is complete.
\end{proof}

\begin{corollary}
\label{enull}If $X$ is a uniform space then every $E^{\ast}$-loop in $X_{E}$
based at $[\ast]_{E}$ is $E^{\ast}$-null.
\end{corollary}

\begin{proof}
If $\widetilde{\alpha}$ is an $E^{\ast}$-loop at $[\beta]_{E}$ then
$\widetilde{\alpha}$ is the unique lift of the $E$-chain $\alpha:=\phi
_{E}(\widetilde{\alpha})$ at $[\beta]_{E}$. By the Chain Lifting Lemma, since
$\widetilde{\alpha}$ is a loop, $[\beta]_{E}=[\alpha\ast\beta]_{E}$, which
implies that $[\alpha]_{E}$ is $E$-null. The Chain Lifing Lemma now shows that
$\widetilde{\alpha}$ is $E^{\ast}$-null.
\end{proof}

When $F\subset E$ are entoruages, the mapping $\phi_{EF}:X_{F}\rightarrow
X_{E}$ described in the Introduction is a special case of the \textit{induced
mapping} defined in \cite{BPUU}: Suppose that $f:X\rightarrow Y$ is a
(possibly not even continuous!) function and $E,F$ are entourages in $X,Y$,
respectively, such that $f(E)\subset F$. Then the function $f_{EF}%
:X_{E}\rightarrow Y_{F}$ defined by $f_{EF}([\alpha]_{E})=[f(\alpha)]_{F}$ is
well-defined, and naturally commutes: $f\circ\phi_{E}=\phi_{F}\circ f_{EF}$.
Put another way, if $\alpha$ and $\beta$ are $E$-homotopic in $X$ then
$f(\alpha)$ and $f(\beta)$ are $F$-homotopic in $Y$. If $X$ and $Y$ are metric
spaces and $f$ is $1$-Lipschitz (distance non-increasing), then as a special
case, used frequently without reference, we have: If $\alpha$ and $\beta$ are
$\varepsilon$-homotopic in $X$ then $f(\alpha)$ and $f(\beta)$ are
$\varepsilon$-homotopic in $Y$.

The Chain Lifting Lemma gives a simple way to describe the natural
identification defined in \cite{BPUU}, Proposition 23, $\iota_{EF}%
:X_{F}\rightarrow\left(  X_{E}\right)  _{F^{\ast}}$ whenever $F\subset E$ are
entourages. If $[\alpha]_{F}\in X_{F}$ then $\iota_{EF}([\alpha]_{F}%
):=[\widetilde{\alpha}]_{F^{\ast}}$, where $\widetilde{\alpha}$ is the unique
lift of $\alpha$ to the basepoint in $X_{E}$. We have the following
commutative diagram:%
\begin{equation}%
\begin{array}
[c]{ccc}%
X_{F} & \overrightarrow{\iota_{EF}} & \left(  X_{E}\right)  _{F^{\ast}}\\
\downarrow\phi_{EF} & \swarrow\phi_{F^{\ast}} & \\
X_{E} &  &
\end{array}
\label{iota}%
\end{equation}

Note that $\iota_{EF}$ also identifies any entourage $D^{\ast}$ in $X_{F}$,
given $D\subset F$, with $\left(  D^{\ast}\right)  ^{\ast}$ in $(X_{E})_{F}$.
This identification of $\phi_{F^{\ast}}$ with $\phi_{EF}$ is useful because it
allows one to apply theorems about $\phi_{E}$ to obtain theorems about
$\phi_{EF}$. As a simple example, one immediately obtains that if $X_{E}$ is
chain connected then $\phi_{EF}:X_{F}\rightarrow X_{E}$ is surjective, because
we already know this about $\phi_{F^{\ast}}:\left(  X_{E}\right)  _{F^{\ast}%
}\rightarrow X_{E}$.

By definition, if $D\subset E\subset F$ are entourages, $\phi_{DE}\circ
\phi_{EF}=\phi_{DF}$ and therefore the collection $\{X_{E},\phi_{EF}\}$ forms
an inverse system, which we referred to in the Introduction as the
\textit{fundamental inverse system} of $X$. When $X$ is metrizable, the
fundamental inverse system has a countable cofinal sequence, which is useful
because elements of the inverse limit $\widetilde{X}$ can be constructed by
iteration. In this case, $\widetilde{X}$ is also metrizable. For any entourage
$E$, we will denote the projection by $\phi^{E}:\widetilde{X}\rightarrow
X_{E}$. Note that $X=X_{E}$ for $E:=X\times X$ and in this case $\phi^{E}$ is
the map $\phi$ from the Introduction. When the projections are surjective,
they are bi-uniformly continuous (by definition of the inverse limit uniformity).

\section{Weakly chained spaces}

Beyond the basic definitions and results, the main goal of this section is
Theorem \ref{prelim}. The primary issue is that $X_{E}$ may not be connected
and therefore the bonding maps $\phi_{EF}:X_{F}\rightarrow X_{E}$ of the
fundamental inverse system may not be surjective. The strategy is to show that
for weakly chained spaces the chain components of $X_{E}$ are always uniformly
open, and these components themselves are weakly chained. Restricting to the
chain components containing the basepoint of all $X_{E}$ then produces a new
inverse system with the same inverse limit $\widetilde{X}$ but with the
advantage that the spaces are chain connected and the bonding maps are
surjective. In the metrizable case this means $\phi$ itself is surjective.

We will denote \textit{arbitrary} chain components of $X_{E}$ by $X_{E}^{K}$,
and the restriction of $\phi_{E}$ to $X_{E}^{K}$ by $\phi_{E}^{K}$. For any
entourage $F\subset E$, we denote $F^{\ast}\cap\left(  X_{E}^{K}\times
X_{E}^{K}\right)  $ by $F^{K}$. Note that the collection of all $F^{K}$ is a
basis for the uniform structure of $X_{E}^{K}$. For the chain component of the
identity we will use \textquotedblleft$c$\textquotedblright\ rather than
\textquotedblleft$K$\textquotedblright, e.g. $\phi_{E}^{c}:X_{E}%
^{c}\rightarrow X$. We let $\pi_{E}^{c}(X)\subset\pi_{E}(X)$ denote the
stabilizer of $X_{E}^{c}$ (i.e. the subgroup that leaves $X_{E}^{c}$ invariant).

\begin{lemma}
\label{chcomp}Let $X$ be a uniform space and $E$ be an entourage. Then
$[\alpha]_{E}\in X_{E}^{c}$ if and only if there are arbitrarily fine chains
$\beta$ such that $[\alpha]_{E}=[\beta]_{E}$.
\end{lemma}

\begin{proof}
That $[\alpha]_{E}\in X_{E}^{c}$ is equivalent to: for every entourage
$F\subset E$ there is an $F^{\ast}$-chain $\widehat{\beta}$ from $[\ast]_{E}$
to $[\alpha]_{E}$. If such a $\widehat{\beta}$ exists then $\widehat{\beta}$
is the unique lift of the $F$-chain $\beta:=\phi_{E}(\widehat{\beta})$, and
since $\widehat{\beta}$ ends at $[\alpha]_{E}$, by the Chain Lifting Lemma
(Lemma \ref{cl}), $[\beta]_{E}=[\alpha]_{E}$. Conversely, if there is such a
$\beta$, then by the Chain Lifting Lemma the lift of $\beta$ to $[\ast]_{E}$
is an $F^{\ast}$-chain that ends at $[\alpha]_{E}$.
\end{proof}

\begin{lemma}
\label{stabil}If $X$ is a uniform space and $F\subset E$ are entourages then
\[
\pi_{E}^{c}(X)=\pi_{E}(X)\cap X_{E}^{c}\text{.}%
\]
In particular, $F^{c}$ is invariant with respect to $\pi_{E}^{c}(X)$.
\end{lemma}

\begin{proof}
Let $[\alpha]_{E}\in X_{E}^{c}$ and $g=[\lambda]_{E}\in\pi_{E}(X)$; that is
$g$ is the uniform homeomorphism of $X_{E}$ induced by pre-concatenation by
$\lambda$. Then $[\alpha]_{E}\in X_{E}^{c}$ if and only if for every entourage
$F\subset E$ there is an $F^{\ast}$-chain $\beta$ from $[\ast]_{E}$ to
$[\alpha]_{E}$. If $g\in\pi_{E}(X)\cap X_{E}^{c}$ then there is an $F^{\ast}%
$-chain $\beta^{\prime}$ from $[\ast]_{E}$ to $[\lambda]_{E}$. Since
$g(F^{\ast})=F^{\ast}$, $g(\beta)$ is an $F^{\ast}$-chain from $[\lambda]_{E}$
to $g([\alpha]_{E})$. Then $\beta^{\prime}\ast\beta$ is an $F^{\ast}$-chain
from $[\ast]_{E}$ to $g([\alpha]_{E})$. Since $F$ was arbitrary,
$g([\alpha]_{E})\in X_{E}^{c}$ and $g\in\pi_{E}^{c}(X)$. The proof of the
opposite inclusion is similar.
\end{proof}

\begin{definition}
We say that an entourage $E$ in a uniform space $X$ is weakly $F$-chained if
there exists some entourage $F\subset E$ such that for every $(x,y)\in F$
there are arbitrarily fine chains $\alpha$ joining $x,y$ such that
$[\alpha]_{E}=[x,y]_{E}$. If $F$ is not specified we simply say that $E$ is
weakly chained. If $F=E$ we say that $E$ is weakly self-chained. We say that
$X$ is weakly chained if $X$ is chain connected and the uniform structure of
$X$ has a basis of weakly chained entourages.
\end{definition}

\begin{remark}
\label{every}Let $D\supset E\supset F\supset G$ be entourages. It is immediate
from the definition that if $F$ is weakly $E$-chained then $G$ is weakly
$D$-chained. In particular, if $F$ is weakly chained then $G$ is weakly
chained. As a consequence, if $X$ is weakly chained then \textit{every}
entourage in $X$ is weakly chained. In the opposite direction, to prove that a
chain connected space is not weakly chained one need only find a single
entourage that is not.
\end{remark}

\begin{remark}
For metric spaces, the above definition is equivalent to the one given in the
Introduction, which is equivalent to the statement that every metric entourage
$E_{\varepsilon}$ is weakly $E_{\delta}$-chained for some $\delta>0$.
\end{remark}

\begin{remark}
\label{PLS}In \cite{PLS} an entourage $E$ was called \textit{chained} if
whenever $(x,y)\in E$ there is a chain connected set $C$ containing $x,y$
contained in $B(x,E)\cap B(y,E)$. This impies that $X_{E}$ is chain connected.
By Lemma 32 in \cite{PLS}, chained entourages are weakly chained. But chained
entourages must have chain connected balls, whereas weakly chained entourages
need not (for example the metric entourages in the Topologist's Sine Curve).
\end{remark}

\begin{remark}
We do not need it for this paper, but Krasinkiewicz proved in 1978 (\cite{K})
that the continuous image of a pointed $1$-movable continuum is pointed
$1$-moveable. According to Remark \ref{pointed}, equivalently the continuous
image of a compact, metrizable weakly chained space is weakly chained. The
statement is false if one removes compactness (e.g. the uniformly continuous
map of $\mathbb{R}$ onto the identity component of the solenoid). It is a nice
exercise to prove Krasinkiewisz's statement directly from the definition of
weakly chained. Hint: if $f(F)\subset E$ then an \textquotedblleft%
$f(F)$-homotopy\textquotedblright\ is an $E$-homotopy even if $f(F)$ is not an
entourage. Begin by showing that for any $\varepsilon>0$, if $y_{i}$ is
sufficiently close to $y$ in the image of $f$ then there are arbitrarily fine
chains $\alpha$ from $y$ to $y_{i}$ such that $[\alpha]_{\varepsilon}%
=[y,y_{i}]_{\varepsilon}$.
\end{remark}

\begin{proposition}
\label{onto}Let $X$ be a (possibly not chain connected) uniform space and $E$
be an entourage that is weakly $F$-chained. Then

\begin{enumerate}
\item Every chain component of $X$ is uniformly $F$-open in $X$.

\item $E^{\ast}$ is weakly $F^{\ast}$-chained in $X_{E}$ and hence any chain
component $X_{E}^{K}$ of $X_{E}$ is uniformly $F^{\ast}$-open.

\item Every $E^{K}$ is weakly $F^{K}$-chained in $X_{E}^{K}$. In particular,
if $X$ is weakly chained then so is $X_{E}^{K}$.

\item If $X_{E}^{K}$ is a chain component satisfying $X_{E}^{K}\cap\phi
_{E}^{-1}(\ast)\neq\varnothing$ then $\phi_{E}^{K}:X_{E}^{K}\rightarrow X$
maps onto $X^{c}$. This is true for every $X_{E}^{K}$ when $X$ is chain connected.
\end{enumerate}
\end{proposition}

\begin{proof}
For the first part, simply note that if $(x,y)\in F$ then by definition $x,y$
are joined by arbitrarily fine chains and so must lie in the same chain
component. For the second, suppose $([\alpha]_{E},[\beta]_{E})\in F^{\ast}$.
By definition of $F^{\ast}$, if $x:=\phi_{E}([\alpha]_{E})$ and $y:=\phi
_{E}([\beta]_{E})$, $(x,y)\in F$, so for any $D\subset F$ there is a $D$-chain
$\alpha$ from $x$ to $y$ such that $[\alpha]_{E}=[x,y]_{E}$. Let
$\widetilde{\alpha}$ be the unique lift of $\alpha$ starting at $[\alpha]_{E}%
$, which is a $D^{\ast}$-chain. Since $[\alpha]_{E}=[x,y]_{E}$, and
$\{[\alpha]_{E},[\beta]_{E}\}$ is the unique lift of $\{x,y\}$ at
$[\alpha]_{E}$, $\widetilde{\alpha}$ must also end at $[\beta]_{E}$. Again by
the Chain Lifting Lemma, $[\widetilde{\alpha}]_{E^{\ast}}=[[\alpha]_{E}%
,[\beta]_{E}]_{E^{\ast}}$. Since $D$ was arbitrary this shows $E^{\ast}$ is
weakly $F^{\ast}$-chained. The last part of the second part now follows from
the first part.

For the third part, suppose $([\alpha]_{E},[\beta]_{E}\in F^{K}$. By the
second part, $E^{\ast}$ is weakly $F^{\ast}$-chained so for every entourage
$D\subset F$ there is a $D^{\ast}$-chain $\widetilde{\alpha}$ from
$[\alpha]_{E}$ to $[\beta]_{E}$ such that $[\widetilde{\alpha}]_{E^{\ast}%
}=[[\alpha]_{E},[\alpha]_{E}]_{E^{\ast}}$. But since $X_{E}^{K}$ is uniformly
$F^{\ast}$-open by the second part, $\widetilde{\alpha}$ stays inside
$X_{E}^{K}$, meaning that is in fact a $D^{K}$-chain, completing the proof.

For the last part, note that since $\phi_{E}$ is uniformly continuous and
$\ast\in C:=\phi_{E}(X_{E}^{K})$, we have that $C\subset X^{c}$. Since $X^{c}$
is by definition chain connected, the proof will be complete if we show that
$C$ is uniformly $F^{con}:=(F\cap\left(  X^{c}\times X^{c}\right)  )$-open in
$X^{c}$. Suppose that $x\in C$ and $(x,y)\in F^{con}$, which by Equation
\ref{surj2} (located just prior to Lemma \ref{cl}) is equal to $\phi
_{E}(F^{\ast})$. Then there is some $x^{\prime}\in X_{E}^{K}$ such that
$\phi_{E}(x^{\prime})=x$, and there are $x^{\prime\prime},y^{\prime\prime}\in
X_{E}$ such that $\phi_{E}(x^{\prime\prime},y^{\prime\prime})=(x,y)$ and
$(x^{\prime\prime},y^{\prime\prime})\in F^{\ast}$. Since $\phi_{E}(x^{\prime
})=\phi_{E}(x^{\prime\prime})=x$, there is some $g\in\pi_{E}(X)$ such that
$g(x^{\prime\prime})=x^{\prime}$. Since $F^{\ast}$ is $g$-invariant,
$y^{\prime}:=g(y^{\prime\prime})$ satisfies $(x^{\prime},y^{\prime})\in
F^{\ast}$ and $\phi_{E}(y^{\prime})=y$. Since $x^{\prime}\in X_{E}^{K}$ and
$X_{E}^{K}$ is uniformly $F^{\ast}$-open by the second part, $y^{\prime}\in
X_{E}^{K}$ and $y\in C$. Finally, suppose $X$ is chain connected and
$X_{E}^{K}$ is an arbitrary chain component containing some $x$. Since $X$ is
chain connected, there is an $F$-chain from $\phi_{E}(x)$ to $\ast$. Since
$X_{E}^{K}$ is uniformly $F^{\ast}$-open, the lift of that chain at $x$, which
is an $F^{\ast}$-chain, stays inside $X_{E}^{K}$. It ends at a point in
$X_{E}^{K}\cap\phi_{E}^{-1}(\ast)$, completing the proof.
\end{proof}

\begin{theorem}
\label{ucon}Let $X$ be a chain connected uniform space, $F\subset E$ be
entourages. The following are equivalent:

\begin{enumerate}
\item $E$ is weakly $F$-chained.

\item The image of any $F^{K}$-ball is an $F$-ball.

\item Every map $\phi_{E}^{K}:X_{E}^{K}\rightarrow X$ is a covering map in
which $F$-balls are evenly covered by unions of $F^{K}$-balls.

\item Every $X_{E}^{K}$ is uniformly $F^{\ast}$-open.
\end{enumerate}
\end{theorem}

\begin{proof}
$1\Rightarrow2$: Let $x^{\prime}\in X_{E}^{K}$ and $x:=\phi_{E}(x^{\prime})$.
By Equation \ref{surj} $\phi_{E}(B(x^{\prime},F^{\ast}))=B(x,F)$, so $\phi
_{E}(B(x^{\prime},F^{K}))\subset B(x,F)$ and we need only show the opposite
inclusion. If $(x,y)\in F=\phi_{E}(F^{\ast})$ then there is some $y^{\prime
}\in B(x^{\prime},F^{\ast})$ such that $\phi_{E}(y^{\prime})=y$. By
Proposition \ref{onto}, $X_{E}^{K}$ is uniformly $F^{\ast}$-open and therefore
$y^{\prime}\in X_{E}^{K}$ and hence $y^{\prime}\in B(x^{\prime},F^{K})$.

$2\Rightarrow3$: Since $X$ is chain connected, we already know that $\phi
_{E}:X_{E}\rightarrow X$ is covering map such that $E$-balls are evenly
covered by unions of $E^{\ast}$-balls. Note that $\left(  \phi_{E}^{K}\right)
^{-1}(B(x,F))$ is the union of the intersections of $F^{\ast}$-balls with
$X_{E}^{K}$. Since $\phi_{E}^{K}$ is surjective by Proposition \ref{onto}, the
only remaining question is whether $\phi_{E}^{K}$ restricted to an $F^{K}%
$-ball is surjective onto an $F$-ball, which is precisely what the second
statement gives us.

$3\Rightarrow4$: Let $x^{\prime\prime}\in X_{E}^{K}$ and suppose that
$(x^{\prime\prime},y^{\prime\prime})\in F^{\ast}$. Letting $x:=\phi
_{E}(x^{\prime\prime})$ and $y:=\phi_{E}(y^{\prime\prime})$, we have by
Equation \ref{surj} that $y\in B(x,F)$. Now the restriction of $\phi_{E}$ to
$B(x^{\prime\prime},F^{K})$ is surjective onto $B(x,F)$ and therefore there is
some $y^{\prime}\in B(x^{\prime\prime},F^{K})$ such that $\phi_{E}(y^{\prime
})=y$. But $\phi_{E}$ is 1-1 on $B(x^{\prime\prime},F^{\ast})$, which means
that $y^{\prime}=y^{\prime\prime}$ and therefore $y^{\prime\prime}\in
X_{E}^{K}$.

$4\Rightarrow1$: Let $(x,y)\in F$. By Equation \ref{surj}, $(x,y)\in\phi
_{E}(F^{\ast})$. This means that there exist $(x^{\prime},y^{\prime})\in
F^{\ast}$ such that $\phi_{E}(x^{\prime})=x$ and $\phi_{E}(y^{\prime})=y$.
Then $x^{\prime}$ lies in some chain component $X_{E}^{K}$, and since
$X_{E}^{K}$ is uniformly $F^{\ast}$-open, $y^{\prime}\in X_{E}^{K}$. Since
$X_{E}^{K}$ is chain connected by definition, for any entourage $D\subset F$
there is a $D^{\ast}$-chain $\beta$ joining $x^{\prime}$ and $y^{\prime}$ in
$X_{E}^{K}$. Then $\phi_{E}(\beta)$ is a $D$-chain joining $x$ and $y$.
Moreover, since $\phi_{E}$ is 1-1 from $F^{\ast}$-balls onto $F$-balls, the
unique lift of $\{x,y\}$ to $x^{\prime}$ in $X_{E}$ is $\{x^{\prime}%
,y^{\prime}\}$, which has the same endpoint as $\beta$. Since $\beta$ is the
unique lift of $\phi_{E}(\beta)$ to $x^{\prime}$, it follows from the Chain
Lifting Lemma that $[x,y]_{E}=[\phi_{E}(\beta)]_{E}$.
\end{proof}

From Theorem \ref{ucon} we immediately obtain:

\begin{corollary}
\label{wcc}Let $X$ be a chain connected uniform space and $E$ be an entourage.
Then the following are equivalent:

\begin{enumerate}
\item $E$ is weakly chained.

\item The chain components of $X_{E}$ are uniformly open.

\item Every map $\phi_{E}^{K}:X_{E}^{K}\rightarrow X$ is a covering map.
\end{enumerate}
\end{corollary}

\begin{corollary}
\label{coverable}Every coverable uniform space is weakly chained.
\end{corollary}

\begin{proof}
Coverable spaces by definition have an uniformity basis such that for every
$E$ in the basis, $\phi^{E}:\widetilde{X}\rightarrow X_{E}$ is surjective,
hence $X_{E}$ is chain connected, and hence $X_{E}^{c}=X_{E}$ and so
$X_{E}^{c}$ is trivially uniformly open in $X_{E}$. The proof is now finished
by Remark \ref{every}.
\end{proof}

\begin{example}
\label{solenoid}It is useful to see why some of these statements fail for the
$2$-adic solenoid $\Sigma$. For this purpose we regard $\Sigma$ as a compact
topological group, namely the inverse limit of circles with their normal group
structure and bonding maps that are double covers (which are also
homomorphisms). As is well-known from the theory of compact, connected groups,
the topology of $\Sigma$ has an open set $U$ of the identity in $\Sigma$ that
is locally isomorphic as a local group to $K\times I$, where $K$ is a Cantor
set and $I$ is an open interval in $\mathbb{R}$. The set $U$ uniquely
determines an invariant entourage $E_{U}$ in $\Sigma$ using the rule $(x,y)\in
E_{U}$ if and only if $xy^{-1}\in U$. As was shown in \cite{BPTG}, this local
isomorphism leads to the fact that $\Sigma_{E_{U}}=K\times\mathbb{R}$.
Intuitively speaking, $\Sigma_{E_{U}}$ is the simplest topological group that
can be reconstructed from relations only contained in $U$, and that obviously
should just be the global product $K\times\mathbb{R}$. At any rate, the chain
components of $\Sigma_{E_{U}}$ are copies of $\mathbb{R}$, and since $K$ has
no isolated points, the chain components cannot be uniformly open. This shows
that $\Sigma$ is not weakly chained. Alternatively, since $\Sigma$ is not path
connected, the restriction of $\phi_{E_{U}}$ to chain components cannot be
surjective onto $\Sigma$. By Proposition \ref{onto}.4, $E_{U}$ is not weakly
chained and therefore $\Sigma$ cannot be weakly chained (Remark \ref{every}).
It is also classically known that $\Sigma$ is not pointed 1-movable (\cite{MS}).
\end{example}

\begin{proposition}
\label{onto2}Suppose that $X$ is a weakly chained uniform space. Then for any
entourages $F\subset E$, the restriction $\phi_{EF}^{c}$ of $\phi_{EF}$ to
$X_{F}^{c}$ is a regular covering map onto $X_{E}^{c}$ with deck group
(naturally isomorphic to) $\pi_{F^{\ast}}^{c}(X_{E})$. This includes the
special case $\phi_{E}^{c}:X_{E}\rightarrow X$ with deck group $\pi_{E}%
^{c}(X)$.
\end{proposition}

\begin{proof}
By \textquotedblleft naturally isomorphic\textquotedblright\ we mean the deck
group $D$ of $\phi_{EF}^{c}$ is equal to the group $R$ of the restrictions to
$X_{F}^{c}$ of the elements of $\pi_{c}^{F}(X)$ in the stabilizer subgroup of
$X_{F}^{c}$. We first prove the stated special case. Clearly $R\subset D$.
Suppose that $h\in D$, i.e. $h:X_{E}^{c}\rightarrow X_{E}^{c}$ is a uniform
homeomorphism such that $\phi_{E}^{c}\circ h=\phi_{E}^{c}$. Let $[\alpha
]_{E}=h([\ast]_{E})\in X_{E}^{c}$. Since $\phi_{E}([\alpha]_{E})=\ast$,
$\alpha$ is in fact an $E$-loop, and therefore represents some $h^{\prime}%
\in\pi_{E}(X)$. But since $[\alpha]_{E}\in X_{E}^{c}$, Lemma \ref{stabil}
implies that $h^{\prime}\in\pi_{E}^{c}(X)$. Since the restriction of
$h^{\prime}$ to $X_{E}^{c}$ is also a deck transformation of $\phi_{E}^{c}$,
by uniqueness of deck transformations $h=h^{\prime}$ on $X_{E}^{c}$. For the
general case we refer to Diagram \ref{iota} (just after Lemma \ref{cl}), which
in this case provides a uniform homeomorphism $\iota_{EF}:X_{F}\rightarrow
\left(  X_{E}\right)  _{F^{\ast}}$. Since all maps are basepoint-preserving,
the restriction of $\iota_{EF}$ to $X_{F}^{c}$ identifies the restriction
$\phi_{EF}^{c}:X_{F}^{c}\rightarrow X_{E}$ with $\phi_{F^{\ast}}^{c}:\left(
X_{E}\right)  _{F^{\ast}}^{c}\rightarrow X_{E}$. In addition, $\iota_{EF}$
identifies the deck group of $\phi_{EF}$ with $\pi_{F^{\ast}}^{c}(X_{E})$. By
Proposition \ref{onto}, $X_{E}^{c}$ is weakly chained and by the special case
we just proved, $\phi_{F^{\ast}}^{c}:\left(  X_{E}\right)  _{F^{\ast}}%
^{c}\rightarrow X_{E}^{c}$ is a regular covering map with deck group
$\pi_{F^{\ast}}^{c}(X_{E})$, completing the proof.
\end{proof}

For a weakly chained uniform space, we now have a new inverse system
$\{X_{E}^{c},\phi_{EF}^{c}\}$ in which the bonding maps $\phi_{EF}^{c}$ are
surjective. The inverse limit $\widehat{X}$ of this new system is a subset of
$\widetilde{X}$.

\begin{theorem}
\label{prelim}If $X$ is a weakly chained uniform space then $\phi
:\widetilde{X}\rightarrow X$ is a generalized regular covering map with deck
group $\pi_{U}(X)=\underleftarrow{\lim}\pi_{E}^{c}(X)$.
\end{theorem}

\begin{proof}
We will show that $\widetilde{X}\subset\widehat{X}$ and hence $\widetilde{X}%
=\widehat{X}$ (this does not require metrizability). For this it suffices to
show that if $\widehat{x}=([\alpha_{E}]_{E})\in\widetilde{X}$ then $\phi
^{E}(\widehat{x})=[\alpha_{E}]_{E}\in X_{E}^{c}$ for every $E$. But by
definition of the inverse limit, for all $F\subset E$, $[\alpha_{F}%
]_{E}=[\alpha_{E}]_{E}$; that is, there are arbitrarily fine chains
$\alpha_{F}$ that are $E$-homotopic to $\alpha_{E}$ and the proof that
$\widehat{X}=\widetilde{X}$ is finished by Lemma \ref{chcomp}, and similarly
$\pi_{U}(X)=\underleftarrow{\lim}\pi_{E}^{c}(X)$. That $\phi$ is a generalized
regular covering map is equivalent to the fact that it can be expressed as the
inverse limit of traditional regular covering maps, see Theorem 44 in
\cite{Pcov} (we only note in this context that in \cite{Pcov} we unfortunately
did not include the term \textquotedblleft regular\textquotedblright\ in our
terminology about generalized covers, which might cause some confusion).
\end{proof}

\begin{remark}
\label{pointed}In \cite{Detal}, Brodskiy-Dydak-LeBuz-Mitra made what amounts
to a simple translation of the Berestovskii-Plaut construction into the
language of Rips complexes, although despite made being aware of this fact
years before publication, the authors unfortunately did not acknowledge this
in their paper. Their construction involves a notion of \textquotedblleft
generalized paths\textquotedblright\ from shape theory of the 1970's due to
Krasinkiewicz-Minc (\cite{KM}). It is natural to consider the space of all
generalized paths starting at a basepoint and imitate the traditional
construction of the universal cover. Why this was not done sooner is not
clear, but evidently the new ingredient, \textquotedblleft inspired
by\textquotedblright\ the Berestovskii-Plaut paper, is to give this space a
uniform structure. However, Section 7 of \cite{Detal}, billed as a
\textquotedblleft comparison\textquotedblright\ with the Berestovskii-Plaut
construction, is inadequate. In fact, their construction is precisely the same
as the Berestovskii-Plaut construction via a simple and natural identification
of elements of $\widetilde{X}$ with generalized paths starting at the
basepoint, as is discussed in the two paragraphs after Example 17 in
\cite{Peq}. In their language, surjectivity of the map $\phi:\widetilde{X}%
\rightarrow X$ is called \textquotedblleft uniform joinable\textquotedblright,
and hence Theorem \ref{prelim} shows that weakly chained, metrizable spaces
are uniform joinable.

Note that both uniform joinable and coverable ultimately require some
information about the fundamental inverse system, while weakly chained does
not, making it substantially easier to verify in some situations.

The main new result in \cite{Detal} is Corollary 6.5, that in the compact
metrizable case, uniform joinable is equivalent to pointed 1-movable. Now an
immediate corollary of Theorem \ref{prelim} is that pointed 1-movable and
weakly chained are equivalent for compact metrizable spaces. Moreover, the
uniform fundamental group $\pi_{U}(X)$ in the compact, weakly chained case is
the classical shape group, and the mappings $\lambda$ and $\kappa$ mentioned
in the Introduction are naturally identified. We note that discrete homotopy
theory is in a sense purely intrinsic to the space, unlike shape theory, which
depends on extrinsic approximations of a space by, or embeddings into, nicer
spaces. For example, the proof of Corollary 6.5 in \cite{Detal} begins with
\textquotedblleft Embed $X$ in the Hilbert Cube $Q$.\textquotedblright

Unfortunately, the main statement of \cite{Peq} about the equivalence of
uniform joinable and coverable has an incorrect proof and this question
remains open. We apologize for any inconvenience this error caused. Brendon
LaBuz has recently independently proved in the metric case that weakly chained
implies uniform joinable (\cite{La}).
\end{remark}

\begin{remark}
It might be interesting to understand better how the concept of weakly chained
fits into classical shape theory beyond the compact case. There is an effort
to generalize some aspects of covering space theory due to Fox (\cite{F}), but
his version of covering spaces, which he refers to as \textquotedblleft
overlays\textquotedblright\ require something like \textquotedblleft evenly
covered\textquotedblright\ (he calls it \textquotedblleft evenly
overlaid\textquotedblright). In particular, an overlay is a local
homeomorphism. But generalized regular covering maps in the current sense need
not be local homeomorphisms--indeed they are inverse limits of traditional
regular covering maps, but the diameters of the evenly covered sets may go to
$0$. For example this happens with the Hawiian Earring and factals (see
\cite{BPUU}, Section 7).
\end{remark}

\section{Additional Tools and Examples}

\begin{example}
\label{length}If $X$ is a length (or intrinsic) metric space, i.e. $d(x,y)$ is
the infimum of the lengths of curves joining $x,y$ for all $x,y\in X$, then
$X$ is sink-free. If not, there exist $x,y\in X$ and $\varepsilon>0$ such that
for all $z\in\overline{B(x,\varepsilon)}$ and $w\in\overline{B(y,\varepsilon
)}$, $d(z,w)\geq d(x,y)$. Since $X$ is a length metric, there exists a curve
$c$ from $x$ to $y$ such that $L(c)<d(x,y)+\varepsilon$. By continuity of the
distance function, there must be points $z\in\overline{B(x,\varepsilon)}$ and
$w\in\overline{B(y,\varepsilon)}$ on $c$. By definition of length, there is
some partition $\mathcal{P}$ of $c$ containing $z$ and $w$ such that the sum
$S$ of the segments determined by $\mathcal{P}$ satisfies $S<L(c)+\varepsilon
$. By the triangle inequality, we have
\[
d(x,y)+2\varepsilon>L(c)+\varepsilon>S\geq d(x,z)+d(z,w)+d(w,y)
\]%
\[
=2\varepsilon+d(z,w)\geq d(x,y)+2\varepsilon\text{,}%
\]
a contradiction.
\end{example}

\begin{example}
\label{gm}Non-geodesic metrics may be sink-free, such as any circle in the
plane with the subspace metric. A square in the plane with the subspace metric
provides an example of a space that is locally sink-free, but not sink-free.
\end{example}

\begin{example}
\label{tsc}The Topologist's Sine Curve $S$ (and its closure) with the subspace
metric from the plane is sink-free. This is easily checked by cases. For
example, suppose that $(x_{1},0)$ and $(x_{2},0)$ both lie in $S$ with
$0<x_{1}<x_{2}$, and the slope of the tangent line at $(x_{2},0)$ is negative.
The slope of the tangent line at $(x_{2},0)$ might also be negative, but will
have strictly larger absolute value. Therefore, moving both points up the
curve (positive $y$-direction) moves them closer to one another.
\end{example}

There are two basic types of $\varepsilon$-homotopies of an $\varepsilon
$-chain $\alpha=\{x_{0},...,x_{n}\}$ to $\{x_{0},x_{n}\}$ when the later is an
$\varepsilon$-chain: \textquotedblleft small ones\textquotedblright\ and
\textquotedblleft lean ones\textquotedblright. \textquotedblleft Small
ones\textquotedblright\ stay inside $B(x_{0},\varepsilon)\cap B(x_{n}%
,\varepsilon)$ and were the basis for the definition of \textquotedblleft
chained entourage\textquotedblright\ in \cite{PLS}. \textquotedblleft Lean
ones\textquotedblright\ may extend far away from $x_{0}$ and $x_{n}$, but
points are in \textquotedblleft opposing pairs\textquotedblright\ that are
closer than $\varepsilon$. Here is the formal definition:

\begin{definition}
Let $0<\tau\leq\varepsilon$ and suppose $\alpha=\{x_{0},...,x_{n}%
,y_{n},...,y_{0}\}$ is a $\tau$-chain in a metric space. Then $\alpha$ is
called $\varepsilon$-lean if for every $i$,
\[
d(x_{i},y_{i}),d(x_{i},y_{i+1}),d(y_{i},x_{i+1})<\varepsilon\text{. }%
\]

\end{definition}

\begin{lemma}
\label{lean}If $\alpha$ is an $\varepsilon$-lean $\tau$-chain from $x$ to $y$
for some $0<\tau\leq\varepsilon$ then $[\alpha]_{\varepsilon}%
=[x,y]_{\varepsilon}$.
\end{lemma}

\begin{proof}
Let $\alpha=\{x_{0},...,x_{n},y_{n},...,y_{0}\}$. We will show by induction
that
\begin{equation}
\lbrack x=x_{0},x_{1}...,x_{j},y_{j},...,y_{0}=y]_{\varepsilon}%
=[x,y]_{\varepsilon} \label{indu2}%
\end{equation}
for all $0\leq j\leq n$. The case $j=0$ is obvious. Suppose that Equation
\ref{indu2} is true for some $0\leq j<n$. By definition of $\varepsilon$-lean
and the fact that $\tau\leq\varepsilon$, the following are legal moves:%
\[
\{x_{0},x_{1},...,x_{j},y_{j},...,y_{0}\}\rightarrow\{x_{0},x_{1}%
,...,x_{j},y_{j+1},y_{j},...,y_{0}\}
\]%
\[
\rightarrow\{x_{0},x_{1},...,x_{j},x_{j+1},y_{j+1},y_{j},...,y_{0}\}\text{.}%
\]

\end{proof}

Note that in the above proof we do not use all three inequalities in the
definition, but we include all three for symmetry.

\begin{proposition}
\label{LSF}Let $X$ be a compact metric space and $\varepsilon>0$. Suppose that
there is some $\sigma>0$ such that if $d(x,y)<\sigma$ in $X$ and $(x,y)$ is a
sink then there are arbitrarily fine chains $\beta$ from $x$ to $y$ such that
$[\beta]_{\varepsilon}=[x,y]_{\varepsilon}$. Then whenever $d(x,y)<\sigma$ in
$X$ there exist arbitrarily fine chains $\alpha$ from $x$ to $y$ such that
$[\alpha]_{\varepsilon}=[x,y]_{\varepsilon}$. In particular, compact locally
sink-free metric spaces are weakly chained.
\end{proposition}

\begin{proof}
Suppose $d(x,y)<\sigma$ and $\tau>0$. Let $S$ be the set of all $t>0$ such
that for some $n$ there are $\tau$-chains $\alpha_{1}=\{x=x_{0},x_{1}%
,...,x_{n}\}$ and $\alpha_{2}=\{y=y_{0},y_{1},...,y_{n}\}$ in $X$ such that
(1) $\alpha_{1}\ast\{x_{n},y_{n}\}\ast\overline{\alpha_{2}}$ is $\varepsilon
$-lean and (2) $d(x_{n},y_{n})<t$. Let $m=\inf S$. We claim that either $m=0$
or there exist $\alpha_{1},\alpha_{2}$ satisfying (1) and $(x_{n},y_{n})$ is a
sink with $d(x,y)<\sigma$. In either case we will be finished. In fact, if
$m=0$ then at some point in (2) we have $t<\tau$, which means that $\alpha
_{1}\ast\{x_{n},y_{n}\}\ast\overline{\alpha_{2}}$ is an $\varepsilon$-lean
$\tau$-chain, hence by Lemma \ref{lean} is $\varepsilon$-homotopic to
$\{x,y\}$. Since $\tau$ is arbitrary, this finishes the proof in this case. If
(1) and (2) hold and $(x_{n},y_{n})$ is a sink, then by assumption
$x_{n},y_{n}$ are joined by a $\tau$-chain $\beta$ such that $[\beta
]_{\varepsilon}=[x_{n},y_{n}]_{\varepsilon}$. But then $\alpha_{1}\ast
\beta\ast\overline{\alpha_{2}}$ is a $\tau$-chain such that $[\alpha_{1}%
\ast\beta\ast\overline{\alpha_{2}}]_{\varepsilon}=[x,y]_{\varepsilon}$,
completing the proof.

To prove the claim, first note that $m\leq d(x,y)$ since if $t>d(x,y)$ then
the chains $\alpha_{1}=\{x\}$ and $\alpha_{2}=\{y\}$ satisfy (1) and (2).
Suppose that $m>0$ and let $t_{j}\searrow m$ with $t_{j}\in S$ and $\alpha
_{1}^{j}=\{x=x_{0}^{j},x_{1}^{j},...,x_{n_{j}}^{j}\},\alpha_{2}^{j}%
=\{y=y_{0}^{j},y_{1}^{j},...,y_{n_{j}}^{j}\}$ are sequences of $\tau$-chains
satisfying (1) and (2) for $t=t_{j}$. Taking a subsequence if necessary we may
assume that $x_{n_{j}}\rightarrow x^{\prime}$ and $y_{n_{j}}\rightarrow
y^{\prime}$ with $d(x^{\prime},y^{\prime})\leq m\leq d(x,y)<\sigma$. For large
enough $j$, $d(x_{n_{j}},x^{\prime}),d(y_{n_{j}},y^{\prime})<\tau$. Therefore
the $\tau$-chains $\alpha_{1}^{\prime}=\{x=x_{0}^{j},...,x_{n_{j}}%
^{j},x^{\prime}\}$ and $\alpha_{2}^{\prime}=\{y=y_{0}^{j},...,y_{n_{j}}%
^{j},y^{\prime}\}$ satisfy (1). Suppose that $(x^{\prime},y^{\prime})$ is not
a sink. Then we may find $x^{\prime\prime},y^{\prime\prime}$ arbitrarily close
to $x^{\prime},y^{\prime}$ such that $d(x^{\prime\prime},y^{\prime\prime
})<d(x^{\prime},y^{\prime})\leq m$. Now $\alpha_{1}^{\prime}\ast\{x^{\prime
},x^{\prime\prime}\}$ and $\alpha_{2}^{\prime}\ast\{y^{\prime},y^{\prime
\prime}\}$ satisfy (1) and (2) with $d(x^{\prime\prime},y^{\prime\prime})<m$,
a contradiction. Therefore $(x^{\prime},y^{\prime})$ must be a sink,
completing the proof of the first statement.

For the last statement, note that if $X$ is LSF($\sigma$) then the hypothesis
of the theorem about $\sigma$ is vaccuous for every $\varepsilon>0$. In other
words, the conclusion of the first statement of the proposition becomes: If
$d(x,y)<\varepsilon<\sigma$ then there are arbitrarily fine chains $\alpha$
from $x$ to $y$ such that $[\alpha]_{\varepsilon}=[x,y]_{\varepsilon}$. Put
another way, every sufficiently small $E_{\varepsilon}$ is weakly
self-chained. Since LSF($\sigma$) requires chain connectivity, the proof of
the last statement is done.
\end{proof}

We will now consider inverse limits$\ X=\underleftarrow{\lim}X_{r}$ of metric
spaces $X_{r}$ indexed on an is an unbounded subset $\Lambda$ of
$\mathbb{R}^{+}$, with surjective, $1$-Lipschitz bonding maps. Since the
indexing set has a countable, totally ordered cofinal set, it follows that the
projection maps $\psi_{r}:X\rightarrow X_{r}$ are also surjective. We will
denote elements of $X=\underleftarrow{\lim}X_{r}$ by $\widehat{x}$ and
$\psi_{r}(\widehat{x})$ by $x_{r}$. Note that \textquotedblleft
sub-Euclidean\textquotedblright\ is stronger than $1$-Lipschitz, so all
results below for $1$-Lipschitz bonding maps are valid for sub-Euclidean
bonding maps.

\begin{lemma}
\label{basis}Consider an inverse system as above. Then

\begin{enumerate}
\item A basis for the inverse limit uniformity on $X=\underleftarrow{\lim
}X_{r}$ consists of the set of all $E_{r,\varepsilon}:=\{(\widehat{x}%
,\widehat{y}):d(x_{r},y_{r})<\varepsilon\}$ for $r,\varepsilon>0$. Moreover,
$E_{r,\varepsilon}\subset E_{s,\delta}$ if $r\geq s$ and $\varepsilon
\leq\delta$.

\item If the bonding maps are sub-Euclidean then for any fixed $K>0$, the set
of all $E_{r,K}$ is a basis for the inverse limit uniformity.
\end{enumerate}
\end{lemma}

\begin{proof}
A standard basis element for the inverse limit uniformity consists of
entourages
\[
E(\varepsilon_{1},...,\varepsilon_{n};r_{1},...,r_{n}):=\{(\widehat{x}%
,\widehat{y}):d(x_{r_{i}},y_{r_{i}})<\varepsilon_{i}\text{ for all
}i\}\text{.}%
\]
Since each $E_{r,\varepsilon}$ is of this form, we need only show that an
arbitrary entourage of the form $E(\varepsilon_{1},...,\varepsilon_{n}%
;r_{1},...,r_{n})$ contains some $E_{r,\varepsilon}$. Let $\varepsilon
:=\min\{\varepsilon_{1},...,\varepsilon_{n}\}$ and $r:=\max\{r_{1}%
,...,r_{n}\}$. If $(\widehat{x},\widehat{y})\in E_{r,\varepsilon}$ then
$d(x_{r},y_{r})<\varepsilon$. By the $1$-Lipschitz assumption, for any $i$,
\[
d(x_{r_{i}},y_{r_{i}})=d(\psi_{r_{i}r}(x_{r}),\psi_{r_{i}r}(y_{r}))\leq
d(x_{r},x_{r})<\varepsilon\leq\varepsilon_{i}\text{.}%
\]
That is, $E_{r,\varepsilon}\subset E(\varepsilon_{1},...,\varepsilon_{n}%
;r_{1},...,r_{n})$. The second part of the first statement is simply a special
case of what we just proved.

Now suppose the bonding maps are sub-Euclidean. We need only show that for any
$E_{r,\varepsilon}$ there is some $E_{s,K}\subset E_{r,\varepsilon}$. But by
the first part we need only note that we may choose $s$ large enough that
$\frac{r}{s}K<\varepsilon$.
\end{proof}

By definition, $\widehat{\gamma}=\{\widehat{x_{0}},...,\widehat{x_{m}}\}$ is
an $E_{r,\varepsilon}$-chain in $X$ if and only if $\gamma_{r}=\{\left(
x_{0}\right)  _{r},...,\left(  x_{m}\right)  _{r}\}$ is an $\varepsilon$-chain
in $X_{r}$. In this circumstance, $\widehat{\gamma}$ will be called a
\textit{lift} of $\gamma$. When $\gamma$ is a loop we will always require that
$\widehat{x_{0}}=\widehat{x_{m}}$ so that any lift $\widehat{\gamma}$ is also
a loop. Note that for $\varepsilon$-chains in $X_{r}$, lifts always exist due
to the surjectivity of the maps $\psi_{r}$. Consider a basic move adding $x\in
X_{r}$ between points $x_{i}$ and $x_{i+1}$ in an $\varepsilon$-chain
$\gamma=\{x_{0},...,x_{m}\}$ in $X_{r}$ which has a given lift
$\widehat{\gamma}=\{\widehat{x_{0}},...,\widehat{x_{m}}\}$. Let $\widehat{x}$
be such that $(\widehat{x})_{r}=x$. Since $d(x_{i},x),d(x,x_{i+1}%
)<\varepsilon$, $(\widehat{x_{i}},\widehat{x}),(\widehat{x},\widehat{x_{i+1}%
})\in E_{r,\varepsilon}$ and in particular, $\{\widehat{x_{0}}%
,...,\widehat{x_{i}},\widehat{x},\widehat{x_{i+1}},...,\widehat{x_{m}}\}$ is
an $E_{r,\varepsilon}$-chain. That is, adding $\widehat{x}$ is a basic move.
The basic move of removing a point $\widehat{x_{i}}$ from $\widehat{\gamma}$
leaves an $E_{r,\varepsilon}$-chain if and only if removing $x_{i}$ leaves an
$\varepsilon$-chain in $X_{r}$. It now follows by induction that if
$\eta=\{\gamma_{0},...,\gamma_{m}\}$ is an $\varepsilon$-homotopy in $X_{r}$
then there are lifts $\widehat{\gamma_{i}}$ of $\gamma_{i}$ such that
$\widehat{\eta}=\{\widehat{\gamma_{0}},...,\widehat{\gamma_{m}}\}$ is an
$E_{r,\varepsilon}$-homotopy. Then $\widehat{\eta}$ will be called a
\textit{lift }of $\eta$. Clearly we can always specify in advance
$\widehat{\gamma_{0}}\in\psi_{\alpha}^{-1}(\gamma_{0})$. What if we have also
specified $\widehat{\gamma_{m}}$ in advance? When the $E_{r,\varepsilon}%
$-homotopy construction above is finished, we have some particular, possibly
different lift $\widehat{\gamma_{m}^{\prime}}=\{\widehat{y_{0}^{\prime}%
},...,\widehat{y_{k}^{\prime}}\}$ of $\gamma_{m}=\{y_{0},...,y_{k}\}$. Note
that since the endpoints in the chains of $\widehat{\eta}$ are never changed,
$\widehat{y_{0}^{\prime}}=\widehat{y_{0}}$ and $\widehat{y_{k}^{\prime}%
}=\widehat{y_{k}}$. Proceeding inductively, observe that for any $i$,
$d((\widehat{y_{i}^{\prime}})_{r},(\widehat{y_{i}})_{r})=d\left(  y_{i}%
,y_{i}\right)  =0<\varepsilon$ and therefore $(\widehat{y_{i}^{\prime}%
},\widehat{y_{i}})\in E_{r,\varepsilon}$. Likewise, $(\widehat{y_{i}%
},\widehat{y_{i+1}^{\prime}})\in E_{r,\varepsilon}$, and we have the following
basic moves:
\[
\{\widehat{y_{0}},...,\widehat{y_{i-1}},\widehat{y_{i}^{\prime}}%
,...,\widehat{y_{k}^{\prime}}\}\rightarrow\{\widehat{y_{0}}%
,...,\widehat{y_{i-1}},\widehat{y_{i}^{\prime}},\widehat{y_{i}}%
,\widehat{y_{i+1}^{\prime}},...,\widehat{y_{k}^{\prime}}\}
\]%
\[
\rightarrow\{\widehat{y_{0}^{\prime}},...,\widehat{y_{i-1}},\widehat{y_{i}%
},\widehat{y_{i+1}^{\prime}},...,\widehat{y_{k}^{\prime}}\}\text{.}%
\]
Therefore we may extend $\widehat{\eta}$ to an $E_{r,\varepsilon}$-homotopy
from $\widehat{\gamma_{0}}$ to $\widehat{\gamma_{m}}$. We will call such a
homotopy \textit{a lift of }$\eta$\textit{\ \textquotedblleft with specified
endpoints\textquotedblright}. To summarize:

\begin{lemma}
\label{homlift}Let $\{X_{r},\psi_{rs}\}_{r,s\in\Lambda}$ be an inverse system
of metric spaces with surjective $1$-Lipschitz bonding maps, where $\Lambda$
is an unbounded subset of $\mathbb{R}$. Suppose that $E_{r,\varepsilon}$ is an
entourage in $X=\underleftarrow{\lim}X_{r}$. If $\eta=\{\gamma_{0}%
,...,\gamma_{k}\}$ is an $\varepsilon$-homotopy in $X_{r}$ then for any choice
of lifts $\widehat{\gamma_{0}},\widehat{\gamma_{k}}$ of $\gamma_{0},\gamma
_{k}$ there is a lift $\widehat{\eta}$ of $\eta$, where $\widehat{\eta}$ is an
$E_{r,\varepsilon}$-homotopy from $\widehat{\gamma_{0}}$ to $\widehat{\gamma
_{k}}$ (i.e. with specified endpoints). In particular, any lift of an
$\varepsilon$-null $\varepsilon$-loop in $X_{r}$ is $E_{r,\varepsilon}$-null
in $X$.
\end{lemma}

\begin{definition}
\label{refine}Let $f:X\rightarrow Y$ be a uniformly continuous surjection
between metric spaces, $0<\delta<\varepsilon$. Then $f$ is said to be $\left(
\varepsilon,\delta\right)  $-refining if whenever $d(a,b)<\delta$ in $Y$, for
all $a^{\prime}\in f^{-1}(a)$ and $b^{\prime}\in f^{-1}(b)$, there are
arbitrarily fine chains $\alpha$ in $X$ from $a^{\prime}$ to $b^{\prime}$ such
that $[f(\alpha)]_{\varepsilon}=\left[  a,b\right]  _{\varepsilon}$. When
$\delta$ exists but is not specified we will simply say that $f$ is
$\varepsilon$-refining. If $f$ is $\varepsilon$-refining for every
$\varepsilon>0$ then $f$ is simply called refining.
\end{definition}

\begin{remark}
\label{ref}Note that if $f:X\rightarrow Y$ is $(\varepsilon,\delta)$-refining
then $E_{\varepsilon}$ is trivially weakly $E_{\delta}$-chained in $Y$.
\end{remark}

\begin{lemma}
\label{cref}Let $f:X\rightarrow Y$ be a uniformly continuous surjection
between metric spaces, $0<\delta<\varepsilon$. Then $f$ is $(\varepsilon
,\delta)$-refining if and only if for every $\delta$-chain $\beta$ in $Y$ from
$a$ to $b$ and $a^{\prime}\in f^{-1}(a)$ and $b^{\prime}\in f^{-1}(b)$, there
are arbitrarily fine chains $\alpha$ in $X$ from $a^{\prime}$ to $b^{\prime}$
such that $[f(\alpha)]_{\varepsilon}=\left[  \beta\right]  _{\varepsilon}$.
\end{lemma}

\begin{proof}
Necessity is obvious. Suppose that $f$ is $(\varepsilon,\delta)$-refining. The
proof is by induction on $n$ for a $\delta$-chain $\beta=\{x_{0},...,x_{n}\}$.
The $n=1$ case is simply the definition of $(\varepsilon,\delta)$-refining.
Suppose the statement is true for a $\delta$-chain $\beta_{i}:=\{x_{0}%
,...,x_{i}\}$ with $0<i<n$ and let $x_{0}^{\prime}\in f^{-1}(x_{0})$,
$x_{i}^{\prime}\in f^{-1}(x_{i})$ and $x_{i+1}^{\prime}\in f^{-1}(x_{i+1})$.
By assumption there are arbitrarily fine chains $\alpha^{\prime}$ from
$x_{0}^{\prime}$ to $x_{i}^{\prime}$ such that $[f(\alpha^{\prime
})]_{\varepsilon}=[\beta_{i}]_{\varepsilon}$. Since $f$ is $(\varepsilon
,\delta)$-refining there are arbitrarily fine chains $\alpha^{\prime\prime}$
from $x_{i}^{\prime}$ to $x_{i+1}^{\prime}$ such that $[f(\alpha^{\prime
\prime})]_{\varepsilon}=[x_{i},x_{i+1}]_{\varepsilon}$. Then $\alpha
:=\alpha^{\prime}\ast\alpha^{\prime\prime}$ is the desired chain.
\end{proof}

\begin{lemma}
\label{compo}If $f:X\rightarrow Y$ and $g:Y\rightarrow Z$ are uniformly
continuous surjective maps such that $h:=g\circ f$ is $\varepsilon$-refining
then $g$ is $\varepsilon$-refining.
\end{lemma}

\begin{proof}
Suppose that $h$ is $(\varepsilon,\delta)$-refining. Let $d(x,y)<\delta$,
$x^{\prime}\in g^{-1}(x)$, and $y^{\prime}\in g^{-1}(y)$. Since $f$ is
surjective there are $x^{\prime\prime}\in f^{-1}(x^{\prime})$ and
$y^{\prime\prime}\in f^{-1}(y^{\prime})$. Since $h$ is $(\varepsilon,\delta
)$-refining, there are arbitrarily fine chains $\alpha$ from $x^{\prime\prime
}$ to $y^{\prime\prime}$ such that $[h(\alpha)]_{\varepsilon}%
=[x,y]_{\varepsilon}$. Since $f$ is uniformly continuous, $f(\alpha)$ is an
arbitrarily fine chain from $x^{\prime}$ to $y^{\prime}$, and since
$[g(f(\alpha))]_{\varepsilon}=[h(\alpha)]_{\varepsilon}=[x,y]_{\varepsilon}$,
the proof is finished.
\end{proof}

\begin{proposition}
\label{compact}Let $f:X\rightarrow Y$ be a $1$-Lipschitz surjection between
compact metric spaces and $\varepsilon>0$. If for all $y\in Y$ and
$x,x^{\prime}\in f^{-1}(y)$ there are arbitrarily fine chains $\alpha$ from
$x$ to $x^{\prime}$ such that $[\alpha]_{\varepsilon}=[x,x^{\prime
}]_{\varepsilon}$ then $f$ is $\varepsilon$-refining. (Note that these
assumptions force $d(x^{\prime},x)<\varepsilon$ for all such $y,x,x^{\prime}$.)
\end{proposition}

\begin{proof}
If $f$ is $(\varepsilon,\delta)$-refining then since $d(y,y)=0$ we may apply
the definition to obtain arbitrarily fine chains $\alpha$ from $x$ to
$x^{\prime}$ isin the definition of Suppose not. Then there exist $y_{i}%
,z_{i}\in Y$, $x_{i}\in f^{-1}(y_{i})$ and $w_{i}\in f^{-1}(z_{i})$ such that
$d(y_{i},z_{i})\rightarrow0$ and there are not arbitrarily fine chains
$\alpha_{i}$ from $x_{i}$ to $w_{i}$ such that $[f(\alpha_{i})]_{\varepsilon
}=[y_{i},z_{i}]_{\varepsilon}$. Taking subsequences if necessary we may assume
that $y_{i},z_{i}\rightarrow y$, $x_{i}\rightarrow x$ and $w_{i}\rightarrow w$
with $x,w\in f^{-1}(y)$. By assumption there are arbitrarily fine chains
$\alpha$ from $x$ to $w$ such that $[\alpha]_{\varepsilon}=[x,w]_{\varepsilon
}$. Let $\tau>0$. For large enough $i$ we have the following: $d(x_{i}%
,w_{i}),d(x_{i},x),d(w_{i},w)<\min\{\tau,\varepsilon\}$. Letting
$\beta:=\{x_{i},x\}\ast\alpha\ast\{w_{i},w\}$, $\beta$ is a $\tau$-chain and
we have $[\beta]_{\varepsilon}=[x_{i},x,w,w_{i}]_{\varepsilon}=[x_{i}%
,w_{i}]_{\varepsilon}$. Since $f$ is $1$-Lipschitz, $[f(\beta)]_{\varepsilon
}=[f(x_{i}),f(w_{i})]_{\varepsilon}=[x_{i},w_{i}]_{\varepsilon}$, a contradiction.
\end{proof}

\begin{proposition}
\label{invpro}Let $\left\{  X_{r},\psi_{rs}\right\}  _{r,s\in\Lambda}$ be an
inverse system of metric spaces with surjective $1$-Lipschitz bonding maps,
where $\Lambda$ is a closed, unbounded subset of $\mathbb{R}^{+}$. Suppose
that $E_{\varepsilon}$ is weakly chained in $X_{r}$ and for all $t>r$ in
$\mathbb{R}^{+}$ there exist $s,s^{\prime}$ such that $r\leq s<t<s^{\prime}$
and both $\psi_{st}$ and $\psi_{ts^{\prime}}$ are $\varepsilon$-refining. Then
$E_{r,\varepsilon}$ is weakly chained in $X$.
\end{proposition}

\begin{proof}
We will show that if $E_{\varepsilon}$ is weakly $E_{\delta}$-chained in
$X_{r}$ then $E_{r,\varepsilon}$ is weakly $E_{r,\delta}$-chained in $X$. Let
$(\widehat{x},\widehat{y})\in E_{r,\delta}$ and consider the following
statement for $t\geq r$. $J(t)$: There are arbitrarily fine chains $\alpha$ in
$X_{t}$ from $x_{t}$ to $y_{t}$ such that $[\psi_{rt}(\alpha)]_{\varepsilon
}=[x_{r},y_{r}]_{\varepsilon}$. If we show that $J(t)$ is true for all $t\geq
r$ then this implies that $E_{r,\varepsilon}$ is weakly $E_{r,\delta}%
$-chained. In fact, any lift $\widehat{\alpha}$ of a $\sigma$-chain $\alpha$
from $x_{t}$ to $y_{t}$ with specified endpoints $\widehat{x}$ and
$\widehat{y}$ is an $E_{t,\sigma}$-chain and also a lift of $\psi_{rt}%
(\alpha)$. By Lemma \ref{homlift}, $[\widehat{\alpha}]_{E_{r,\varepsilon}%
}=[\widehat{x},\widehat{y}]_{E_{r,\varepsilon}}$.

Note that $J(r)$ is true because $d(x_{r},y_{r})<\delta$, and $E_{\varepsilon
}$ is weakly $E_{\delta}$-chained in $X_{r}$. Let $T:=\sup\{t:J(t)$ is
true$\}$; we will show $T=\infty$. We will first show that $J(t)$ implies
$J(t+u)$ for some $u>0$. In fact, by assumption, for some $u>0$, $\psi
_{t,t+u}$ is $(\varepsilon,\delta)$-refining for some $\delta>0$. Since $J(t)$
is true there is a $\delta$-chain $\beta$ in $X_{t}$ from $x_{t}$ to $y_{t}$
such that $[\psi_{rt}(\beta)]_{\varepsilon}=[x_{r},y_{r}]_{\varepsilon}$.
Since $\psi_{t,t+u}$ is $(\varepsilon,\delta)$-refining there exist
arbitrarily fine chains $\alpha$ from $x_{t+u}$ to $y_{t+u}$ such that
$[\psi_{t,t+u}(\alpha)]_{\varepsilon}=[\beta]_{\varepsilon}$. By the
$1$-Lipschitz assumption (for the second equality below),
\[
\lbrack\psi_{r,t+u}(\alpha)]_{\varepsilon}=[\psi_{rt}(\psi_{t,t+u}%
(\alpha))]_{\varepsilon}=[\psi_{rt}(\beta)]_{\varepsilon}=[x_{r}%
,y_{r}]_{\varepsilon}\text{.}%
\]

This shows in particular that $T>r$, and the proof will now be complete if we
show that if $T<\infty$ then $J(T)$ is true. By assumption there is some
$r\leq s<T$ such that $\psi_{sT}$ is $(\varepsilon,\tau)$-refining for some
$\tau>0$. Since $t<T$ there is some $\tau$-chain $\beta$ in $X_{t}$ such that
$[\psi_{rt}(\beta)]_{\varepsilon}=[x_{r},y_{r}]_{\varepsilon}$. Now let
$\alpha$ be an arbitrarily fine chain in $X_{T}$ from $x_{T}$ to $y_{T}$ such
that $[\psi_{tT}(\alpha)]_{\varepsilon}=[\beta]_{\varepsilon}$ (Lemma
\ref{cref}). Since $\psi_{rt}$ is $1$-Lipschitz, $\psi_{rT}(\alpha)$ is also
an arbitrarily fine chain such that
\[
\lbrack\psi_{rT}(\alpha)]_{\varepsilon}=[\psi_{rt}(\psi_{tT}(\alpha
))]_{\varepsilon}=[\psi_{rt}(\beta)]_{\varepsilon}=[x_{r},y_{r}]_{\varepsilon
}\text{.}%
\]

\end{proof}

The above proposition gives a general approach to show that the inverse limit
of weakly chained spaces is weakly chained. We state one theorem that applies
when the bonding maps are $1$-Lipschitz (used in the current paper) and
sub-Euclidean (used in \cite{PCAT}).

\begin{theorem}
\label{invlim}Let $\left\{  X_{r},\psi_{rs}\right\}  _{r,s\in\Lambda}$ be an
inverse system of chain connected metric spaces with surjective $1$-Lipschitz
(resp. sub-Euclidean) bonding maps, where $\Lambda$ is a closed, unbounded
subset of $\mathbb{R}^{+}$. Suppose that for every $\varepsilon>0$ there
exists some $R>0$ (resp. there exist some $R,K>0$) such that for all $t>r\geq
R$ there exist $s,s^{\prime}$ such that $r\leq s<t<s^{\prime}$ and both
$\psi_{st}$ and $\psi_{ts^{\prime}}$ are $\varepsilon$-refining (resp.
$K$-refining). Then $X=\underleftarrow{\lim}X_{r}$ is weakly chained.
\end{theorem}

\begin{proof}
Since each $X_{r}$ is chain connected and the bonding maps are surjective, the
fact that $X$ is chain connected follows from Lemma 11 in \cite{BPUU}. Now for
any $r>R$, by the hypothesis of the theorem and Remark \ref{ref},
$E_{\varepsilon}$ (resp. $E_{K}$) is weakly chained in $X_{r}$. Now it follows
from Proposition \ref{invpro} that the basis element $E_{r,\varepsilon}$
(resp. $E_{r,K}$) is weakly chained, completing the proof.
\end{proof}

\begin{proof}
[Proof of Theorem \ref{CAT}]For more background CAT(0) spaces, see \cite{BH},
Chapter II.8. Let $X$ be a CAT(0) space and $x_{0}\in X$; denote the (unique)
geodesic from $x_{0}$ to $x$ by $\gamma_{x}$. For $r\leq s$, one has the
projections $p_{rs}:\overline{B}(x_{0},s)\rightarrow\overline{B}(x_{0},r)$
between closed metric balls defined as follows: If $d(x_{0},x)>r$ then
$p_{rs}(x)=\gamma_{x}(r)$; otherwise (i.e. $x\in\overline{B}(x_{0},r)$),
$p_{rs}(x)=x$. These projections are the bonding maps for an inverse system,
the inverse limit of which is denoted by $\overline{X}$. By definition,
elements of $\overline{X}$ are contained in $%
{\displaystyle\prod\limits_{r\in\mathbb{R}^{+}}}
\overline{B}(x_{0},r)$ and denoted by $\widehat{x}=(x_{r})$. There is a
topological embedding of $X$ into $\overline{X}$ defined as follows
$\iota(x)=\widehat{x}$, where $x_{r}=x$ whenever $r\geq d(x_{0},x)$. The
boundary $\partial X$ at $x_{0}$ (which is topologically independent of
$x_{0}$) is defined to be $\overline{X}\backslash\iota(X)$.
\end{proof}

Now consider the restrictions $\psi_{rs}$ of $p_{rs}$ to $\Sigma_{x_{0}}(s)$.
We again have an inverse system $(\Sigma_{r}(x_{0}),\psi_{rs})$, and by the
CAT(0) condition, $\psi_{rs}$ is sub-Euclidean. We claim that
$B:=\underleftarrow{\lim}\Sigma_{r}(x_{0})=\partial X$. First note that since
each $\Sigma_{x_{0}}(r)$ is contained in $\overline{B}(x_{0},r)$,
$B\subset\overline{X}$. Moreover, since elements of $B$ do not have constant
coordinates, $B\subset\partial X$. Therefore we need only show the opposite
inclusion. There is a natural bijection between elements of $\partial X$ and
unit parameterized geodesic rays starting at $x_{0}$, which takes a geodesic
ray $\gamma$ to the element of $\widehat{x}\in\partial X$ with $x_{r}%
=\gamma(r)$. But $x_{r}\in\Sigma_{r}(x_{0})$ by definition, showing that
$\partial X\subset B$. Note that in the geodesically complete case, the
projections $\psi_{rs}$ are surjective onto $\Sigma_{r}(x_{0})$. Therefore the
assumptions in the first sentence of Theorem \ref{invlim} are satisfied for
the inverse system $(\Sigma_{x_{0}}(r),\psi_{rs})_{r\in\mathbb{R}}$.

Now let $R$ be such that if $t\geq R$ then the remaining hypotheses of the
theorem hold. Suppose that $t>r\geq R$, $0<t-s\leq\frac{\min\{\iota
(t),t-r\}}{2}$, and $y\in\Sigma_{x_{0}}(s)$. Let $x,x^{\prime}\in\left(
\psi_{st}\right)  ^{-1}(y)$. By definition of the projection and the triangle
inequality, $d(x,x^{\prime})\leq\iota(t)$ and by assumption, if $(x,x^{\prime
})$ is not a sink then there is a curve from $x$ to $x^{\prime}$ in
$\Sigma_{x_{0}}(t)\cap B(x,K)\cap B(x^{\prime},K)$. But then arbitrarily fine
chains on this curve are $K$-homotopic to $\{x^{\prime},x\}$ and hence by
Proposition \ref{compact} $\psi_{st}$ is $K$-refining. On the other hand, for
$s^{\prime}>t$ close enough to $t$, $s^{\prime}-t<\frac{\iota(s^{\prime})}{2}$
and the same argument as above finishes the proof.

\begin{example}
We will revisit the solenoid $\Sigma$ discussed earlier, to see why the
hypotheses of Theorem \ref{invlim} fail in this case--as they must since
$\Sigma$ is not weakly chained. Considering $\Sigma$ as the inverse limit of
circles $C_{i}$, we may give the circles their standard Riemannian metrics,
with the diameter of the $i^{th}$ circle equal to $2^{i}$; that is, the double
covers $\phi_{i,i+1}:C_{i+1}\rightarrow C_{i}$ are local isometries and hence
$1$-Lipschitz. However, the double covers are not $\varepsilon$-refining for
small enough $\varepsilon$. To simplify the notation, simply consider the
double cover $\phi:C^{\prime}\rightarrow C$, where $C$ has circumference $1$
and $C^{\prime}$ has circumference $2$. For points $x_{1},x_{2}\in C$, let
$x_{1}^{\prime}$ be one of the two points in $\phi^{-1}(x_{1})$ and
$x_{2}^{\prime}$ be the point in $\phi^{-1}(x_{2})$ closest to the antipodal
point of $x_{1}^{\prime}$. If $d(x_{1},x_{2})<\varepsilon=\frac{1}{3}$ then
$d(x_{1}^{\prime},x_{2}^{\prime})>\frac{2}{3}$. If $\alpha$ is an arbitrarily
fine chain between $x_{1}^{\prime}$ and $x_{2}^{\prime}$, its image must wrap
more than $2/3$ of the way around the circle, and therefore cannot be
$\varepsilon$-homotopic to its endpoints. It is not hard to see that this must
be true for small enough $\varepsilon>0$--a basic move cannot
\textquotedblleft cross\textquotedblright\ the circle. The number $\frac{1}%
{3}$ more precisely is the single \textquotedblleft homotopy critical
value\textquotedblright\ of $C$ (see \cite{PW1}). This is, $\frac{1}{3}$
precisely the largest $\varepsilon$ at which the $\varepsilon$-cover of $C$
\textquotedblleft unrolls\textquotedblright\ into a line.
\end{example}

\section{Remaining properties of $\phi:\protect\widetilde{X}\rightarrow X$}

In \cite{BPUU}, the first step to fully understand $\phi:\widetilde{X}%
\rightarrow X$ in the coverable case was to prove $\widetilde{X}$ is
\textquotedblleft universal\textquotedblright\ which is equivalent to
$\widetilde{X}$ being coverable with $\pi_{U}(X)$ trivial (Corollary 52,
\cite{BPUU}). However, the argument depends heavily on the fact that $X$ is
coverable, and we see no way to modify it to the weakly chained case (and
indeed we do not know whether $\widetilde{X}$ is coverable when $X$ is weakly
chained). Therefore we follow a different strategy that uses Theorem
\ref{invlim}.

\begin{proposition}
\label{ueq}If $X$ is a weakly chained metrizable uniform space then the
following are equivalent:

\begin{enumerate}
\item $X$ is uniformly simply connected.

\item For every entourage $E$, $\pi_{E}^{c}(X)$ is trivial (equivalently
$\phi_{E}^{c}:X_{E}^{c}\rightarrow X$ is a uniform homeomorphism).

\item $\phi:\widetilde{X}\rightarrow X$ is a uniform homeomophism.
\end{enumerate}
\end{proposition}

\begin{proof}
That the first and third parts are equivalent follows from Theorem
\ref{prelim}: $\phi$ is a quotient map via the group $\pi_{U}(X)$. If $\pi
_{U}(X)$ is trivial then again by Theorem \ref{prelim}, every $\theta^{E}%
:\pi_{U}(X)\rightarrow\pi_{E}^{c}(X)$ is surjective, so $\pi_{E}^{c}(X)$ is
trivial. Conversely, if every $\pi_{E}^{c}(X)$ is trivial then $\pi
_{U}(X)=\underleftarrow{\lim}\pi_{E}^{c}(X)$ is trivial.
\end{proof}

\begin{theorem}
If $X$ is a weakly chained metrizable uniform space then $\widetilde{X}$ is
uniformly simply connected.
\end{theorem}

\begin{proof}
To see that $\widetilde{X}$ is weakly chained, we will use Theorem
\ref{invlim} and consider $\widetilde{X}$ as $\underleftarrow{\lim}X_{r}$,
where $X_{r}:=X_{E_{r}}^{c}$. Given any metric on $X$, each $X_{\varepsilon
}:=X_{E_{r}}$ has a natural \textquotedblleft lifted metric\textquotedblright%
\ (\cite{PW1}, Definition 12, Proposition 13-14), with $d([\alpha]_{r}%
,[\beta]_{r})$ defined to be the infimum of the lengths $L(\kappa)$ of
$r$-chains $\kappa$ with $[\kappa]_{r}=[\alpha\ast\overline{\beta}]_{r}$,
where
\begin{equation}
L(\{x_{0},...,x_{n}\})=\sum_{i=1}^{n}d(x_{i},x_{i-1})\geq d(x_{0}%
,x_{n})\text{.} \label{met}%
\end{equation}
This metric is an isometry on $\frac{r}{2}$-balls and is invariant with
respect to the action of $\pi_{r}(X)$. While not stated in \cite{PW1}, it
follows from Inequality (\ref{met}) and Diagram (\ref{iota}) that $\phi
_{rs}:X_{s}\rightarrow X_{r}$ is $1$-Lipschitz. Now fix $\varepsilon>0$
consider any $r>0$. By Proposition \ref{onto}, for some $0<s<r$, $E_{r}$ is
weakly $E_{s}$-chained in $X$. Suppose that $d([\alpha]_{r},[\beta]_{r})<s$ in
$X_{r}$. Setting $x:=\phi_{r}([\alpha]_{r})$ and $y:=\phi_{r}([\beta]_{r})$,
since $s<r$, $d(x,y)=d([\alpha]_{r},[\beta]_{r})<s$. That is, there are
arbitrarily fine chains $\gamma$ from $x$ to $y$ such that $[\gamma
]_{r}=[x,y]_{r}$. By uniqueness in the Chain Lifting Lemma and the fact that
$\phi_{r}$ is a bijection from $s$-balls to $s$-balls, the lift $\gamma^{\ast
}$ of $\gamma$ is an arbitrarily fine chain from $[\alpha]_{r}$ to
$[\beta]_{r}$ such that $[\gamma^{\ast}]_{r}=[[\alpha]_{r},[\beta]_{r}]_{r}$.
Now suppose that $a\in\left(  \phi_{rs}^{c}\right)  ^{-1}([\alpha]_{r})$ and
$b\in\left(  \phi_{rs}^{c}\right)  ^{-1}([\beta]_{r})$. Since the restriction
of $\phi_{rs}^{c}$ is a bijection on $s$-balls there is some $b^{\prime}%
\in\phi_{rs}^{-1}([\beta]_{r})\cap B(a,s)$. By the Chain Lifting Lemma,
$\gamma^{\ast}$ lifts to an arbitrarily fine chain $\gamma^{\ast\ast}$ from
$a$ to $b^{\prime}$. Since $b,b^{\prime}\in X_{s}^{c}$ there are arbitrarily
fine chains $\gamma^{\ast\ast\ast}$ from $b$ to $b^{\prime}$ in $X_{s}^{c}$.
But $\phi_{rs}(\gamma^{\ast\ast\ast})$ is an arbitrarily fine loop in $X_{r}$
and hence is $r$-null by Corollary \ref{enull}. That is, $\phi_{rs}%
(\gamma^{\ast\ast}\ast\gamma^{\ast\ast\ast})$ is $r$-homotopic to
$\{[\alpha]_{r},[\beta]_{r}\}$. We have shown that if $s<r$ is positive and
sufficiently small, $\phi_{rs}^{c}$ is $\varepsilon$-refining. Now suppose
that $s<r$ is arbitrary and positive. By what we just proved, for some
sufficiently small positive $t<r$, $\phi_{rt}^{c}:X_{t}\rightarrow X_{r}$ is
$\varepsilon$-refining. It now follows from Corollary \ref{compo} that
$\phi_{rs}^{c}$ is also $\varepsilon$-refining. That is, whenever $0<r<s$,
$\phi_{rs}^{c}$ is $\varepsilon$-refining. Theorem \ref{invlim} now applies
(strictly speaking we need to reverse the indexing), showing that
$\widetilde{X}$ is weakly chained.

To finish the proof of the theorem, note that the set of all $G:=\left(
\phi^{E}\right)  ^{-1}(F^{\ast})=\left(  \phi^{E}\right)  ^{-1}(F^{\ast})$,
where $E$ is weakly $F$-chained, is a basis for $\widetilde{X}$. But a word of
caution: the terminology $F^{c}$ does not indicate in which space $X_{D}^{c}$
the entourage lives, but we will always mention the space in what follows. We
first claim that if $(x,y)\in G$ then there is some $z\in\widetilde{X}$ such
that $\{x,z,y\}$ is a $G$-chain and $\phi^{F}(\{x,y,z\})$ is an $F^{c}$-chain
in $X_{F}$. By definition of $G$, $\left(  \phi^{E}(x),\phi^{E}(y)\right)  \in
F^{c}$ in $X_{E}$. Since $\phi_{EF}^{c}$ is a bijection on $F^{c}$-balls,
there exists some $z_{F}\in B(\phi^{F}(x),F^{c})$ in $X_{F}^{c}$ such that
$\phi_{EF}(z_{F})=\phi^{E}(y)$. Let $z$ be any lift of $z_{F}$ to $G$. Since
$\phi^{E}(z)=\phi_{EF}(z_{F})=\phi^{E}(y)$ then $(\phi^{E}(z),\phi
^{E}(y))=(\phi^{E}(y),\phi^{E}(y))$ is trivially in $F^{c}$ and so $(y,z)\in
G$. Likewise, $(\phi^{E}(z),\phi^{E}(x))=(\phi^{E}(x),\phi^{E}(y))\in F^{c}$,
so $(x,z)\in G$. Now note that since $(x,y)\in G$, $[x,z,y]_{G}=[x,y]_{G}$. By
iteration we obtain the following: If $\alpha$ is any $G$-chain in
$\widetilde{X}$ then there is a $G$-chain $\beta$ in $\widetilde{X}$ such that
$[\alpha]_{G}=[\beta]_{G}$ and $\phi^{F}(\beta)$ is an $F^{c}$-chain in
$X_{F}$. In particular, we may always assume that if $[\lambda]_{G}\in\pi
_{G}(\widetilde{X})$ then $\phi^{F}(\lambda)$ is an $F^{c}$-loop in $X_{F}$
and hence $F^{\ast}$-null by Corollary \ref{enull}.

Identifying $X_{F}$ with $\left(  X_{E}\right)  _{F^{\ast}}$ via Diagram
\ref{iota} and note that the entourage $(F^{\ast})^{\ast}$ in $\left(
X_{E}\right)  _{F^{\ast}}$ corresponds to the entourage $F^{\ast}$ in $X_{F}$.
Therefore the $F^{\ast}$-null loop $\phi^{F}(\lambda)$ corrresponds to an
$(F^{\ast})^{\ast}$ loop $\lambda^{\prime}$ in $(X_{E})_{F^{\ast}}$. By
Equation (\ref{surj2}), $\phi^{E}(\lambda)=\phi_{EF}(\phi^{F}(\lambda
))=\phi_{F^{\ast}}(\lambda^{\prime})$ is $F^{c}$ null in $X_{E}$. Now any
$F^{c}$-null homotopy of $\phi^{E}(\lambda)$ lifts to a $G$-null homotopy of
$\lambda$. That is, $\pi_{G}(\widetilde{X})$ is trivial and the proof that
$\pi_{U}(\widetilde{X})$ is trivial is complete by Proposition \ref{ueq}.
\end{proof}

The next theorem is a stronger version of the Lifting property than was stated
in the Introduction, in that the target space need not be weakly chained.

\begin{theorem}
[Strong Lifting]Let $X$ be uniformly simply connected and $Y$ be any uniform
space. For every uniformly continuous function $f:X\rightarrow Y$ there is a
unique basepoint-preserving uniformly continuous function $f_{L}%
:X\rightarrow\widetilde{Y}$ such that $f=\phi\circ f_{L}$.
\end{theorem}

\begin{proof}
Let $E$ be an entourage in $Y$. Since $f$ is uniformly continuous there is an
entourage $F$ in $X$ such that $f(F)\subset E$ and therefore the induced
mapping $f_{EF}:X_{F}\rightarrow Y_{E}$ is defined and uniformly continuous,
and is the unique base-point preserving map such that $\phi_{Y}\circ
f_{EF}=f\circ\phi_{X}$ (Theorem 30, \cite{BPUU}). By Proposition \ref{ueq},
$\phi_{F}^{c}:X_{F}^{c}\rightarrow X$ is a uniform homeomorphism and we may
define $f_{L}^{E}:=f_{EF}\circ\left(  \phi_{F}^{c}\right)  ^{-1}$. Since
$\phi_{Y}\circ f_{EF}=f\circ\phi_{X}$, $f_{L}^{E}:X\rightarrow Y_{E}$ is a
lift of $f$. Note that if $D\subset E$ is an entourage in $Y$, then by
definition, $f_{L}^{E}=\phi_{ED}\circ f_{L}^{D}$ and therefore by the
universal property of inverse limits there is a unique uniformly continuous
lift $f_{L}:X\rightarrow\widetilde{Y}$ defined by $\phi^{E}(f_{L}%
(y))=f_{L}^{E}(y)$.
\end{proof}

Having established the the \textbf{Lifting }property and the fact that
$\widetilde{X}$ is uniformly simply connected, the proofs of the
\textbf{Universal} and \textbf{Functoria}l properties are essentially the same
as those of Theorems 2 and 3, respectively, in \cite{BPUU}.

\begin{example}
\label{TC}The Texas Circle $T$, defined by Labuz (\cite{La}), consists of the
graph of $f(x)=sin2x+1/x$ for $x\in\lbrack\pi,\infty)$, the same interval of
the $x$-axis, and the vertical segment from $(\pi,0)$ to $(\pi,\frac{1}{\pi}%
)$. As was pointed out in \cite{La} and is easy to check, this space is not
weakly chained with the subspace uniformity. However, the fine uniformity is
simply the uniform structure of the Euclidean metric on the line, which of
course is weakly chained. This example shows that it is possible for a
metrizable weakly chained topological space to have a uniformity that is not
weakly chained. It remains an open question whether a path connected
metrizable space with the fine uniformity must be a weakly chained uniform
space--a positive answer to which would provide a more straightforward path to
the UU-cover of a path connected metric space than the one described in the introduction.
\end{example}

\begin{proposition}
\label{loc2g}If $X$ is a locally chain connected (resp. locally connected,
locally path connected) topological space then with the fine uniformity then
$X$ is uniformly locally chain connected (resp. uniformly locally connected,
uniformly path connected) in the sense that there is a basis for the
uniformity such that every entourage $E$ in the basis has chain connected
(resp. connected, path connected) balls. In particular, if $X$ is chain
connected then and satisfies any of these local conditions then with the fine
uniformity it is coverable.
\end{proposition}

\begin{proof}
The arguments for all three cases are essentially the same so we only consider
the chain connected case. Let $F$ be an entourage and $G$ be an entourage such
that $G^{2}\subset F$. Then for every $x\in X$ there is some chain connected
open set $U_{x}$ such that $x\in U_{x}\subset B(x,G)$. Define $E$ to be the
union of all sets $U_{x}\times U_{x}\subset X\times X$. If $(x,y)\in E$ then
$(x,y)$ lies in some $U_{z}\times U_{z}$, meaning that $x\in B(G,z)$ and $y\in
B(G,z)$, so $(x,y)\in G^{2}\subset F$. So $E\subset F$ and since $F$ was
arbitrary, the set of all such $E$ is a basis. The same argument shows that
$B(x,E)$ is the union of all $U_{z}$ such that $x\in U_{z}$. That is, $B(x,E)$
is a union of chain connected sets that all include $x$, and hence is chain
connected. For the last statement, it follows from Proposition 71,
\cite{BPUU}, that in all three cases every such $X_{E}$ is chain connected and
hence $X_{E}=X_{E}^{c}$, which is trivially uniformly open in $X_{E}$.
Therefore, if $X$ is weakly chained then $E$ is a weakly chained entourage by
Theorem \ref{ucon}.
\end{proof}

\begin{acknowledgement}
I appreciate useful conversations I with Jeremy Brazas, Brendon LaBuz, Mike
Mihalik, and Kim Ruane. This paper is dedicated to Betsy Saylor-Plaut.
\end{acknowledgement}

\end{document}